\documentclass[11pt]{article}
\usepackage[utf8]{inputenc}
\usepackage{amscd,amsmath,amsthm,amssymb,amsfonts}
\usepackage{txfonts}
\usepackage{eucal}
\usepackage{bbm}
\usepackage{xcolor}
\usepackage{appendix}
\usepackage{authblk}
\usepackage{graphicx} 
\usepackage{hyperref}
\hypersetup{
    colorlinks=true,
    linkcolor=blue,
    filecolor=blue,      
    urlcolor=blue,
    citecolor=blue,
    }
\usepackage{algorithm}
\usepackage{mathrsfs}
\usepackage{derivative}
\usepackage[noend]{algorithmic}
\usepackage[a4paper]{geometry}
\usepackage{tikz-cd}
\usepackage{mathtools}
 \usepackage[nottoc]{tocbibind}    
\usepackage{empheq}
\usepackage{graphics}
\usepackage{indentfirst}
\usepackage[
backend=bibtex,
style=numeric-comp,
maxnames=4
]{biblatex}
\usepackage{color}
\usepackage{tabularx,colortbl}
\numberwithin{equation}{section}

\usepackage{fancyhdr}

\usepackage{subcaption}
\DeclareCaptionLabelFormat{custom}
{%
      \textbf{#1 #2}
}
\DeclareCaptionLabelSeparator{custom}{ }
\DeclareCaptionFormat{custom}
{%
    #1 #2 #3
}
\usepackage{etoolbox}

\makeindex

\def\s{$\mbox{\c{s}}$}

\def\NZQ{\varmathbb}
\def\N{{\NZQ N}}

\def\R{{\NZQ R}}

\def\opn#1#2{\def#1{\operatorname{#2}}} 

\opn\Ker{Ker} \opn\Coker{Coker}  \opn\Hom{Hom} \opn\Im{Im}
\opn\End{End} \opn\Aut{Aut} \opn\defect{def} \opn\ord{ord}
\opn\id{id} \opn\dim{dim} \opn\det{det} \opn\tr{tr} \opn\grad{grad} \opn\lcm{lcm}
\opn\min{min} \opn\max{max} 
\opn\Span{Span}   \opn\rang{rang}  \opn\id{id} \opn\Ass{Ass} \opn\Min{Min}
\opn\GL{GL} \opn\SL{SL} \opn\mod{mod} \opn\diag{diag}
\opn\min{min} \opn\sgn{sgn} \opn\ini{in_<}  \opn\Mon{Mon} \opn\LC{LC_<} \opn\Hom{Hom} \opn\Ext{Ext} \opn\gini{gin_{<_{rev}}} \opn\gin{gin_{<}}
\opn\LT{LT_<}
\opn\s{supp} \opn\Tor{Tor} \opn\link{link} \opn\depth{depth} \opn\pd{pd} \opn\reg{reg} 

\newcommand{\der}{\textup{d}}

\date{}
\title{On an inhomogeneous coagulation model with a differential sedimentation kernel}
\author[a,1]{Iulia Cristian}
\author[a,2]{Barbara Niethammer}
\author[a,3]{Juan J. L. Vel\'{a}zquez}
\affil[a]{Institute for Applied Mathematics, University of Bonn, Endenicher Allee 60, 53115 Bonn, Germany}
\affil[1]{\href{mailto:cristian@iam.uni-bonn.de}{cristian@iam.uni-bonn.de}}
\affil[2]{\href{mailto:niethammer@iam.uni-bonn.de}{niethammer@iam.uni-bonn.de}}
\affil[3]{\href{mailto:velazquez@iam.uni-bonn.de}{velazquez@iam.uni-bonn.de}}

\begin{document}
\maketitle
\begin{abstract}
We study an inhomogeneous coagulation equation that contains a transport term in the spatial variable modeling the sedimentation of clusters. We prove local existence of mass conserving solutions for a class of coagulation kernels for which in the space homogeneous case instantaneous gelation (i.e., instantaneous loss of mass) occurs. Our result holds true in particular for sum-type kernels of homogeneity greater than one, for which solutions do not exist at all in the spatially homogeneous case. Moreover, our result covers kernels that in addition vanish on the diagonal, which have been used to describe the onset of rain and the behavior of air bubbles in water.


\end{abstract}


\tableofcontents
\newtheorem{teo}{Theorem}[section]
\newtheorem{ex}[teo]{Example}
\newtheorem{prop}[teo]{Proposition}
\newtheorem{obss}[teo]{Observations}
\newtheorem{cor}[teo]{Corollary}
\newtheorem{lem}[teo]{Lemma}
\newtheorem{prob}[teo]{Problem}
\newtheorem{conj}[teo]{Conjecture}
\newtheorem{exs}[teo]{Examples}
\newtheorem{alg}[teo]{\bf Algorithm}
\newtheorem{assumption}[teo]{Assumption}
\theoremstyle{definition}
\newtheorem{defi}[teo]{Definition}

\theoremstyle{remark}
\newtheorem{rmk}[teo]{Remark}
\newtheorem{ass}[teo]{Assumption}

\parindent0mm




\section{Introduction}
\subsection{Background}

In \cite{physicspaper} the authors suggest a space-dependent coagulation equation to model the onset of rain. Here spherical particles of volume $v$ move in space vertically, for example due to gravitation, and merge when their trajectories cross. This leads to the following inhomogeneous coagulation equation for the density $f$ of  particles of size $v$ at the point $x$:
\begin{align}\label{rain and air bubbles}
    \partial_{t}f(x,v,t)+v^{\frac{2}{3}}\partial_{x}f(x,v,t)&=\frac{1}{2}\int_{(0,v)}K(v-v',v')f(x,v-v',t)f(x,v',t)\der v'\nonumber\\
    &\quad -\int_{(0,\infty)}K(v,v')f(x,v,t)f(x,v',t)\der v'
\end{align}
with a  so-called differential sedimentation  kernel of the form
\begin{align}\label{rain kernel}
    K(v,v')=|v^{\frac{2}{3}}-v'^{\frac{2}{3}}|(v^{\frac{1}{3}}+v'^{\frac{1}{3}})^{2}.
\end{align}
This choice of kernel is motivated by the following consideration (see \cite{physicspaper}):  the cross-section of interaction between two particles of radii $r$ and $r'$, and volume $v$ and $v'$, respectively, that merged upon touching is given by  $\pi(r+r')^{2}\approx (v^{\frac{1}{3}}+v'^{\frac{1}{3}})^{2}$.

Additionally, the velocity is approximately $v^{\frac{2}{3}}$, which represents the Stokes velocity of a rigid sphere with no slip boundary condition, and the collision rate between two particles is taken to be proportional to their relative velocities $|v^{\frac{2}{3}}-v'^{\frac{2}{3}}|$.

The model is used to describe the behavior of air bubbles in water which move due to buoyancy and it is also valid for water droplets. We refer to \cite{raininitiation, bubblesdropsparticles, precipitation} for more details. In \cite{physicspaper}, slip-flow corrections for water droplets are discussed and it is mentioned that this requires to change the power of the volume for the velocity. More precisely, the left-hand side of \eqref{rain and air bubbles} becomes $\partial_{t}f(x,v,t)+v^{\alpha}\partial_{x}f(x,v,t)$, for some $\alpha\in(0,1)$, and the kernel in \eqref{rain kernel} has the form $K(v,v')=|v^{\alpha}-v'^{\alpha}|(v^{\frac{1}{3}}+v'^{\frac{1}{3}})^{2}$.

The model in (\ref{rain and air bubbles}) with kernel \eqref{rain kernel} is referred to as the \textit{free merging regime} in \cite{physicspaper},  since it is assumed in its derivation that the particles merge when their trajectories cross. When studying equations like \eqref{rain and air bubbles}, it is customary to look for stationary solutions of non-zero flux (cf. \cite{tanaka}) of the form $f\approx v^{d}$, for some $d\in\mathbb{R}$. One of the possible approaches is to compute them using the so-called Zakharov transform (see \cite{zakharov}) and using it we find the solution $f(x,v,t)\approx v^{-\frac{13}{6}}$. However, this approach can be made rigorous only if the integral containing the coagulation kernel in \eqref{rain and air bubbles} is finite, which is not the case for the kernel \eqref{rain kernel}. In order to be able to rigorously find a solution for kernels with the same homogeneity, the so-called \textit{forced locality regime} in which only particles of similar sizes can merge is studied in \cite{physicspaper}. More precisely, for the \textit{forced locality regime} a cut-off in the coagulation kernel is introduced that makes the kernel vanish outside the region where $\frac{1}{q}<\frac{v'}{v}<q$, for some $q>1$. With this cut-off, the integral containing the coagulation kernel converges and thus the stationary solution $f\approx v^{-\frac{13}{6}}$ is a valid solution.

Our main goal in this paper is to show that mass conserving solutions exist, at least for a short time interval, for a class of inhomogeneous coagulation equations that includes example \eqref{rain and air bubbles} with \eqref{rain kernel}. At a first glance this might look surprising since the homogeneity of the kernel in \eqref{rain kernel} is greater than one. Indeed, it  is well-known that gelation (mass loss) occurs for the standard one-dimensional coagulation equation, 
\begin{align}\label{non-existence model}
    \partial_{t}f(v,t)=\frac{1}{2}\int_{(0,v)}K(v-v',v')f(v-v',t)f(v',t)\der v'-\int_{(0,\infty)}K(v,v')f(v,t)f(v',t)\der v',
\end{align}
when the coagulation kernel behaves like a power law of homogeneity $\gamma>1$ (see, for example \cite{gelationpaper,esclaumischperth,papergelationlaurencot}). In particular, for sum kernels of the form
\begin{align}\label{sum kernel}
K(v,v')=v^{\gamma}+v'^{\gamma},
\end{align}
with $\gamma>1$, gelation happens instantaneously. Actually, making use of this property, one can prove that solutions which belong to $L^1$ for the standard coagulation equation do not exist at all for kernels as in \eqref{sum kernel} (see \cite{ballcarr,Carr1992,Dongen,bookcoagulation}). 
In addition, it has been proven in \cite{instangelvanish} that the instantaneous gelation phenomenon holds for Radon measure solutions of the standard coagulation equation with sum kernels of homogeneity greater than one which vanish on the diagonal, i.e. $K(v,v')=0$, such as the kernel in \eqref{rain kernel}.

Our main goal is then  to prove that, in contrast to the homogeneous case, there exist, at least for short times, mass conserving solutions to the inhomogeneous model
\begin{align}\label{lala}
\partial_{t}f(x,v,t)+v^{\alpha}\partial_{x}f(x,v,t)&=\frac{1}{2}\int_{(0,v)}K(v-v',v')f(x,v-v',t)f(x,v',t)\der v'\nonumber\\
    &\qquad -\int_{(0,\infty)}K(v,v')f(x,v,t)f(x,v',t)\der v',
\end{align}
where $\alpha \in (0,1)$. Our proof holds for a rather general class of coagulation kernels (see Assumption \ref{assinhomogeneous}), in particular kernels of the form (\ref{rain kernel}) and (\ref{sum kernel}). 
Thus, the model (\ref{lala}) provides a coagulation model in which existence for kernels of the form  (\ref{sum kernel}) with $\gamma>1$ holds, at least for short times.

\subsection{Main result}

\subsubsection*{Short time existence of mass conserving solutions for the inhomogeneous model}

Our goal is to prove short-time existence of mass conserving solutions for the inhomogeneous model 
\begin{align}\label{rain and air bubbles general}
\partial_{t}f(x,v,t)+v^{\alpha}\partial_{x}f(x,v,t)&=\frac{1}{2}\int_{(0,v)}K(v-v',v')f(x,v-v',t)f(x,v',t)\der v'\nonumber\\
    &\qquad -\int_{(0,\infty)}K(v,v')f(x,v,t)f(x,v',t)\der v',
\end{align}
where
\begin{align}\label{condition on alpha}
    \alpha\in(0,1).
\end{align}
\begin{assumption}\label{assinhomogeneous}
We assume that $K\colon [0,\infty) \times [0,\infty) \to [0,\infty)$ is a symmetric and continuous function that satisfies \begin{align}\label{condition on gamma existence}
0\leq K(v,v')\leq K_{1}(v^{\gamma}+v'^{\gamma}), \textup{ with } \gamma\in[0,1+\alpha)
\end{align}
for some constant $K_{1}>0$ and
\begin{align}\label{growth condition on the kernel}
 K(v-v',v')\leq K(v,v'), \textup{ when } v'\in\big[0,\tfrac{v}{2}\big].
\end{align}
\end{assumption}

Condition (\ref{growth condition on the kernel}) is a rather standard assumption in the study of coagulation equations, see for example \cite{diffusion}, and most of the kernels used in applications satisfy this condition, in particular,  kernels of the form $K(v,v')=v^{\gamma}+v'^{\gamma}$ or (\ref{rain kernel}). The condition $\gamma<1+\alpha$ in (\ref{condition on gamma existence}) is such that the transport term will control the contribution coming from the coagulation term.

\begin{defi}\label{mild solution definition}[Mild solutions] Let $\alpha\in(0,1)$, $\gamma\in[0,1+\alpha)$ and $m>\frac{\gamma+1}{\alpha}$. Let $T>0$ and $K$ satisfy Assumption \ref{assinhomogeneous}. We say that a non-negative function $f\in\textup{C}([0,T]\times\mathbb{R}\times (0,\infty))$ such that
\begin{align}\label{time small makes integral bounded}
   \sup_{t\in[0,T]} \int_{(0,\infty)}(1+v^{\gamma})f(x,v,t)\der v &\leq \frac{C_T}{\max\Big\{1,|x|^{m-\frac{\gamma+1}{\alpha}}\Big\}}, \textup{ for }x\in\mathbb{R},
\end{align}
is a mild solution of equation (\ref{rain and air bubbles general}) if
\begin{align}\label{mild solution equation}
&f(x,v,t)-f(x-v^{\alpha}t,v,0)  S[f](x,v,0,t)=\\
&\frac{1}{2}\int_{0}^{t}\int_{(0,v)}S[f](x,v,s,t)
K(v-v',v')f(x-(t-s)v^{\alpha},v-v',s) f(x-(t-s)v^{\alpha},v',s)\der s,\nonumber 
\end{align}
for all $t\in[0,T]$, $v\in(0,\infty)$  and $x \in\mathbb{R}$, where
\begin{align}\label{definition s mild}
  S[f](x,v,s,t):=  \textup{e}^{-\int_{s}^{t}a[f](x-v^{\alpha}(t-\tau),v,\tau)\der\tau},
\end{align}
with
\begin{align}\label{general a f def}
a[f](x,v,t):=\int_{(0,\infty)}K(v,v')f(x,v',t)\der v'.
\end{align}
\end{defi}
\begin{defi}\label{mass cons defi mild sol}
   We call $f\in\textup{C}([0,T]\times\mathbb{R}\times(0,\infty))$ a mass-conserving solution of equation (\ref{rain and air bubbles general}) if $f$ is as in Definition \ref{mild solution definition} and satisfies in addition
   \begin{align*}
              \int_{\mathbb{R}}\int_{(0,\infty)}vf(x,v,t)\der v\der x= \int_{\mathbb{R}}\int_{(0,\infty)}vf(x,v,0)\der v\der x
   \end{align*}
   for all $t\in[0,T]$.
\end{defi}
\begin{teo}[Local existence  of solutions]\label{main theorem} Let $\alpha\in(0,1)$, $\gamma\in[0,1+\alpha)$, $m\in\mathbb{N}$ even, and $p=\alpha m$ with $m>\max\{\frac{2\gamma+1}{\alpha}, \frac{2}{\alpha}+3\}$. Let $K$ satisfy Assumption \ref{assinhomogeneous} and  $T>0$ be sufficiently small. Let $f_{\textup{in}}\in\textup{C}^{1}(\mathbb{R}\times(0,\infty))$ such that
\begin{align*}
    f_{\textup{in}}(x,v)\leq \frac{C_0}{1+|x|^{m}+v^{p}},
\end{align*}
for some $C_0>0$ and all $x\in\mathbb{R}$, $v\in (0,\infty)$. Then there exists a mass-conserving solution $f$ of (\ref{mild solution equation})  as in Definition \ref{mass cons defi mild sol} that satisfies 
\begin{align}\label{see decay}
     f(x,v,t)\leq \frac{C}{1+|x|^{m}+v^{p}}, 
\end{align}for all $t\in[0,T]$, for some $C>0$.
\end{teo}
\begin{rmk}
Theorem \ref{main theorem} is valid for coagulation kernels $K$ as in (\ref{rain kernel}), as well as coagulation kernels of the form $K(v,v')=v^{\gamma}+v'^{\gamma}.$
\end{rmk}

\begin{rmk}
    It is worthwhile to mention that mass conservation will follow due to the fact that our solution will have sufficiently fast decay for large values of $|x|$ and $v$, see \eqref{see decay}. For more details, see the proof of Theorem \ref{main theorem} .
\end{rmk}

When $\gamma<1$, our result could be expected according to the general theory of existence for one-dimensional coagulation equations. This states that solutions exist for kernels that behave like power laws of homogeneity $\gamma<1$, see, for example,  \cite{stewart} for existence of solutions and \cite{articlefournier} for existence of self-similar profiles.

Some multi-dimensional coagulation models have been studied in the mathematical literature, see \cite{multicom1, cristian}.  Moreover, several classes of coagulation models for the distribution of particles with space dependence have also been considered. In particular, models in which in addition to coagulation there is space diffusion of the aggregating particles can be found in \cite{aizenbak,canizo,carrillo,diffusion}. Models that contain coagulation of particles as well as transport terms (that might include sedimentation terms) were studied in \cite{dubovskiichae,dubovskii,dubovskiiparttwo,dubovskiichaeparttwo,burobin,galkinbounded}. In all the models mentioned above the homogeneity of the coagulation kernel is either $\gamma<1$, case in which the solutions are globally defined and preserve the total mass, or product-type kernels are discussed, for which solutions preserve the mass up to a certain point in time.


To our knowledge, the only exception that considers the case $\gamma>1$ for space-dependent models is \cite{galkin}. Indeed,  existence of solutions for the discrete version of the model in \eqref{rain and air bubbles} has been established for coagulation kernels of the form $K(i,j)=\sigma_{ij}|\nu_{i}-\nu_{j}|,$ $i,j\in\mathbb{N}$. Here, $\nu_{i}$ is a non-negative function of volume which represents the sedimentation velocity of the particles and $\sigma_{ij}$ can be estimated by a function of the form $C(i^{\gamma}+j^{\gamma})$ with $\gamma<1$. Assuming that $\nu_{i}$ scales like a power law $i^{\beta}$, we would obtain a coagulation kernel $K(i,j)$ that behaves like a homogeneous function with homogeneity $\beta+\gamma$ that might be larger than one. As a matter of fact, the results in \cite{galkin} hold for any choice of $\nu_{i}\geq 0$. In particular, the solutions constructed in \cite{galkin} have total mass of particles that changes continuously in time. Nevertheless it has not been proven in \cite{galkin} that the mass of solutions is conserved. By contrast, in the present paper we construct a theory of existence for mild solutions which conserve mass in finite time. The class of kernels considered in \cite{galkin} has a non-empty intersection with the class of kernels considered in this paper and this intersection includes the coagulation kernel (\ref{rain kernel}), which is relevant in the physical literature as explained above.  However, none of two classes of kernels is included in the other. In addition, the arguments in \cite{galkin} seem to rely significantly on the discrete character of the coagulation process and it is not clear how difficult would be to extend these results to the continuous case. On the other hand, the sedimentation $\nu_{i}$ can be arbitrary in \cite{galkin}. This opens the question of whether we can extend our result to powers of $\alpha$ which are larger than one for the model in \eqref{rain and air bubbles general}. However, a different approach of obtaining information from the characteristics created by the transport term in \eqref{rain and air bubbles general} will be needed in this case.

 Moreover, kernels of the form $K(v,v')=v^{\gamma}+v'^{\gamma},$ with $\gamma\in(1,1+\alpha)$, also satisfy the conditions (\ref{condition on gamma existence}), (\ref{growth condition on the kernel}). Solutions of  the standard one-dimensional coagulation equation do not exist for these types of kernels, see \cite{ballcarr,Carr1992,Dongen,bookcoagulation}. This is to our knowledge the first result of existence of mass-conserving solutions involving sum kernels of homogeneity $\gamma>1$, regardless of whether one considers one-dimensional or multi-dimensional coagulation models. 

It is worth mentioning that our results also include the cases $\gamma=0$ and $\gamma=1$, which are normally studied separately in the literature due to their rich features, such as being able to predict the long-time behavior of solutions, see \cite{menonpego, menonpegodynamical,tranvan,mouhot,bonacini}, or prove uniqueness of self-similar profiles, see \cite{menonpego, throm1, throm2}, where the constant kernel and perturbations of the constant kernel are discussed.

The strategy for proving existence of solutions is to consider an iterative scheme based on  a linear version of (\ref{rain and air bubbles general}). For this equation, we are able to find a suitable supersolution, which in turn will provide sufficiently good moment estimates. This will give us compactness of the iterated sequence and enable us to pass to the limit in the equation. The idea of finding appropriate supersolutions  was also used in \cite{niethammer}, in this case for finding self-similar solutions with fat tails. The idea of considering a linear version of the model in order to better study properties of its solutions is common in the study of coagulation equations, see for example \cite{throm1} where this idea is used to study uniqueness of solutions.

We present the proof of our main Theorem \ref{main theorem} in the following Section \ref{existence inhomogeneous model section} with some technical computations moved to the appendix.

     \section{Proof of the main theorem}\label{existence inhomogeneous model section}

 \subsection{Formal approximation and discussion on the assumptions}\label{appendix a}
 
 Our approach to prove existence of a solution  to  (\ref{rain and air bubbles general}) is based on constructing a suitable supersolution by approximating the coagulation term for large particles by a transport term.
 To motivate this we present in this subsection this formal approximation of \eqref{rain and air bubbles general}.
  Similar computations can be found in \cite[Section 4 and Appendix 1]{physicspaper}.
  
  Suppose now that $f$ is a solution of \eqref{rain and air bubbles general}.
Since we are interested in the behavior for large values of $v$, we can assume, due to the fast decay of $f(x,v,t)$, that the term $\int_{[\frac{v}{2},\infty)}K(v,v')f(x,v,t)$ $f(x,v',t)\der v'$ gives a small contribution. This is consistent with the known results in coagulation equations.  We can then approximate (\ref{rain and air bubbles general})  via 
\begin{align}
&\partial_{t}f(x,v,t)+v^{\alpha}\partial_{x}f(x,v,t) \nonumber\\
&=\int_{(0,\frac{v}{2})}K(v-v',v')f(x,v-v',t)f(x,v',t)\der v'\nonumber-\int_{(0,\infty)}K(v,v')f(x,v,t)f(x,v',t)\der v'
    \\
    &=\int_{(0,\frac{v}{2})}[K(v-v',v')f(x,v-v',t)-K(v,v')f(x,v,t)]f(x,v',t)\der v'\nonumber\\
    &\qquad -\int_{[\frac{v}{2},\infty)}K(v,v')f(x,v,t)f(x,v',t)\der v'\nonumber\\
    &\approx\int_{(0,\frac{v}{2})}[K(v-v',v')f(x,v-v',t)-K(v,v')f(x,v,t)]f(x,v',t)\der v'\label{first part approximation formal}.
    \end{align}

Since our strategy relies on finding a suitable supersolution, it suffices to find a lower bound for \eqref{first part approximation formal}. This is where the assumption (\ref{growth condition on the kernel}) is needed. We thus use that $K(v-v',v')\leq K(v,v')$ when $v'\in[0,\frac{v}{2}]$ and obtain that
\begin{align}
\partial_{t}&f(x,v,t)+v^{\alpha}\partial_{x}f(x,v,t) \nonumber\\
    &-\int_{(0,\frac{v}{2})}[K(v-v',v')f(x,v-v',t)-K(v,v')f(x,v,t)]f(x,v',t)\der v'\nonumber\\
    &\geq \partial_{t}f(x,v,t)+v^{\alpha}\partial_{x}f(x,v,t) -\int_{(0,\frac{v}{2})}K(v,v')[f(x,v-v',t)-f(x,v,t)]f(x,v',t)\der v'.\label{this is where the assumption is needed}
\end{align}
    Assuming now that the coagulation kernel $K$ behaves like $v'^{\gamma}+v^{\gamma}$ and since $v'\in[0,\frac{v}{2}]$, we further deduce that
    \begin{align}\label{approximation derivation}
\partial_{t}f(x,v,t)+v^{\alpha}\partial_{x}f(x,v,t)  &\approx\int_{(0,\frac{v}{2})}v^{\gamma}[f(x,v-v',t)-f(x,v,t)]f(x,v',t)\der v'\nonumber\\
&\approx - v^{\gamma}\int_{(0,\frac{v}{2})}\int_{v-v'}^{v}\partial_{w}f(x,w,t)\der w f(x,v',t)\der v'.
\end{align}
Assume that $\partial_{w}f(x,w,t)$ behaves similarly for $w\in[\frac{v}{2},v]$ and thus
\begin{align}\label{approximate model with derivative}
\partial_{t}f(x,v,t)+v^{\alpha}\partial_{x}f(x,v,t)   \approx-v^{\gamma}\partial_{v} f(x,v,t)\int_{(0,\frac{v}{2})}v'f(x,v',t)\der v'.
\end{align}
We denote by
$M_1(x,t):=\int_{(0,\infty)}v'f(x,v',t)\der v'$ the first moment in $v$ of $f$.
We consider only large values of $v$ so that we can safely assume that $\int_{(0,\frac{v}{2})}v'f(x,v',t)\der v'$ contains most of the mass. In this manner, we can further approximate \eqref{approximate model with derivative} by
\begin{align}\label{final step approximate model}
\partial_{t}f(x,v,t)+v^{\alpha}\partial_{x}f(x,v,t)\approx -v^{\gamma}\partial_v f(x,v,t) M_1(x,t).
\end{align}

Notice that in order for our approximation to hold, the assumption (\ref{growth condition on the kernel}) was needed in \eqref{this is where the assumption is needed}. Otherwise, an analogous approximation of the model could be obtained by replacing $v^{\gamma}\partial_v f(x,v,t)$ in \eqref{approximate model with derivative} by $\partial_v\big(v^{\gamma}f(x,v,t)\big)$. The approximation containing the term $\partial_v\big(v^{\gamma}f(x,v,t)\big)$ is the one actually used in \cite{physicspaper}. However, due to (\ref{growth condition on the kernel}), the approximation used in (\ref{final step approximate model}) suffices in order to prove our desired result.
Suppose now that $M_1(x,t)$ decays sufficiently fast for large values of $x$, that is 
assume that 
\begin{align}\label{moment x t decays sufficiently fast}
    M_1(x,t)\leq \frac{L}{1+|x|^{\overline{m}}},
\end{align}
for some sufficiently large $\overline{m}$ and some $L>0$. Combining (\ref{moment x t decays sufficiently fast}) with (\ref{final step approximate model}), we obtain that $f$ should behave formally like the solution of the equation
\begin{align}\label{approximate model}
\partial_{t}f(x,v,t)+v^{\alpha}\partial_{x}f(x,v,t)   +\frac{
    Lv^{\gamma}\partial_v f(x,v,t)}{1+|x|^{\overline{m}}}=0.
\end{align}
This motivates the analysis of the equation (\ref{eq char}) below when trying to find a supersolution for a linear version of (\ref{rain and air bubbles general}). In order to obtain a behavior of the form (\ref{moment x t decays sufficiently fast}) for $M_1(x,t)$, we have to work with functions $f$ such that \eqref{see decay} holds.

 \subsection{Upper and lower bounds for the solution of the approximated model}
To prove short time existence of a solution to  (\ref{rain and air bubbles general}) we will set up an iterative scheme and derive, using \eqref{approximate model}, a uniform supersolution for the solutions of this scheme. More precisely, for $n\in\mathbb{N}$, we define inductively a sequence of functions $\{f_{n}\}_{n\in\mathbb{N}}$  as follows:
\begin{align}\label{inductive sequence}
    \partial_{t}f_{n+1}(x,v,t)+v^{\alpha}\partial_{x}f_{n+1}(x,v,t)&=\frac{1}{2}\int_{(0,v)}K(v-v',v')f_{n+1}(x,v-v',t)f_{n}(x,v',t)\der v'\nonumber\\
    &\qquad-\int_{(0,\infty)}K(v,v')f_{n+1}(x,v,t)f_{n}(x,v',t)\der v',
\end{align}
with $\alpha\in(0,1)$, and
\begin{align}\label{initial function}
f_{n+1}(x,v,0)=f_{n}(x,v,0)= f_{\textup{in}}(x,v)\leq \frac{C_0}{1+|x|^{m}+v^{p}}, \qquad \textup{ for all } n\in\mathbb{N}.
\end{align}

We take $f_{0}$ to be a function such that
\begin{align}\label{definition f0}
  \partial_{t}f_{0}(x,v,t)+v^{\alpha}\partial_{x}f_{0}(x,v,t)=0
\end{align}
with
\begin{align}\label{definition f0 initial condition}
  f_{0}(x,v,0) =f_{\textup{in}}(x,v) \leq\frac{C_0}{1+|x|^{m}+v^{p}}.
\end{align}
Notice that, if we have equality in (\ref{definition f0 initial condition}), the solution of (\ref{definition f0}) is  $f_{0}(x,v,t)=\frac{C_0}{1+|x-v^{\alpha}t|^{m}+v^{p}}$.

\begin{rmk}
 In principle we have to prove that the  sequence in (\ref{inductive sequence}) is well-defined. A rigorous proof would work as follows. First we approximate $K$ with a kernel $K_{N}$ which is such that 
\begin{align}\label{truncated kernel}
K_{N}(v,v')=K(v,v')\chi_{N}(v+v'),
\end{align}
for $N>1$ and where $\chi_{N}:[0,\infty)\rightarrow[0,1]$ is a continuous functions such that $\chi_{N}(x)=1$, when $x\leq \frac{N}{2}$, and $\chi_{N}(x)=0,$ when $x\geq N$. Then we can establish with a standard fixed-point argument the existence and uniqueness of $f_{n,N}$. The key result is then that we obtain a uniform  fast decaying bound for the sequence of solutions
which is in particular independent of $N$. Then one can pass to the limit $N \to \infty$. Since this procedure is standard once one has the bounds on the solution, we omit the details here and work directly with $K$.
\end{rmk}

{\bf Notations and Assumptions:} 
Let $R>0$. In the following we will denote by 
\begin{equation}\label{chirdef}
\xi_{R}\in C([0,\infty))\,, \quad \xi_R\colon [0,\infty)\rightarrow [0,1] \; \mbox{ such that } \xi_{R}(v)=1
\mbox{ if } v\geq 2R \mbox{ and } \xi_{R}(v)=0 \mbox{ if } v\leq R.
\end{equation}
Furthermore we assume
\begin{equation}\label{parameters}
 \alpha \in (0,1)\,, \quad \gamma \in [0,1+\alpha)\,, \quad m \in \N\,,\quad m \mbox{ even}\,,  \quad p=\alpha m\,, \quad   
 m>\max\Big\{\frac{2\gamma+1}{\alpha}, \frac{2}{\alpha}+3\Big\}\,,
\end{equation}
and define $d$ via
  \begin{equation}\label{we need function to be even}
\left\{\begin{aligned}
d&=\bigg[\frac{2}{\alpha}\bigg]+1; &\textup{ if }\bigg[\frac{2}{\alpha}\bigg]& \textup{ odd; }\\
d&=\bigg[\frac{2}{\alpha}\bigg]+2;&\textup{ if }\bigg[\frac{2}{\alpha}\bigg]& \textup{ even, }
   \end{aligned}\right.
   \end{equation}
   where $[\cdot ]$ denotes the floor function. Note that \eqref{parameters} and \eqref{we need function to be even} in particular imply that $m>d+1$.

\bigskip
The main goal in this section will be to derive estimates for the solution of a transport equation that approximates the coagulation equation. 
 For $L>0$ let  $G_{L}$ be the solution of 
\begin{align}
\partial_{t}G_{L}(x,v,t)+v^{\alpha}\partial_{x}G_{L}(x,v,t)+\frac{{L}v^{\gamma}}{1+|x|^{m-d}}\xi_{R}(v)\partial_{v}G_{L}(x,v,t)&=0,\label{eq char}\\
G_{L}(x,v,0)&=\frac{C_0}{1+|x|^{m}+v^{p}}\,.\nonumber
\end{align}
We first study the properties of the  backward characteristics for  equation (\ref{eq char}). To this end, we define   $X$ and $V$ via 
    \begin{equation}\label{the original characteristics}
\left\{\begin{aligned}
\partial_{t}{X}(x,v,t)&=-{V}^{\alpha}, & \qquad {X}(x,v,0)&=x\,, \\
\partial_{t}{V}(x,v,t)&=-\frac{{L}{V}^{\gamma}\xi_{R}({V})}{1+|{X}|^{m-d}},& \qquad {V}(x,v,0)&=v\,,
   \end{aligned}\right.
   \end{equation}
 where $d$ was defined in (\ref{we need function to be even}) and ${L}$ is as in \eqref{eq char}. 

\begin{prop}[Properties of the characteristics]\label{properties of the characteristics}
Given $L>0$ and $\delta\in\big(0,\frac{1}{2}\big)$ there exists a sufficiently large $R>0$ such that 
 for all $t\geq 0$ the following estimates hold:
\begin{align}
(1-\delta) v &\leq V(x,v,t) \leq  (1+\delta) v ;\label{Vestimate}\\
(1-\delta) v^{\alpha}t &\leq x-X(x,v,t) \leq  (1+\delta)  v^{\alpha} t ;\label{Xestimate}\\
 \frac{\alpha}{18} v^{\alpha -1} t &\leq -\partial_{v} X(x,v,t) \leq 18
 \alpha v^{\alpha-1} t.
 \label{dvXestimate}
 \end{align}
 Moreover, if $x\not\in [(1-2\delta) v^{\alpha}t, (1+2\delta) v^{\alpha}t]$, then
 \begin{align}
 \frac{1}{4} & \leq \partial_{v} V(x,v,t) \leq \frac{9}{4}. \label{dvVestimate}
\end{align}
Otherwise, if $x\in[(1-2\delta) v^{\alpha}t, (1+2\delta) v^{\alpha}t]$, then 
 \begin{align}
 |\partial_{v} V(x,v,t)| &\leq 36 L \max\{1,v^{\gamma-1}t\}; \label{dvVestimate problematic region}\\
 \partial_{v} V(x,v,t)&\leq 2.\label{dvVestimate decreasing}
\end{align}
 \end{prop}

 \begin{cor}
  As an immediate consequence of \eqref{Xestimate} we obtain the following estimates for all $t\geq 0$.
 
 If $x\leq 0$ then 
  \begin{equation}\label{xnegative}
 |x|+(1-\delta) v^{\alpha} t = |x-(1-\delta) v^{\alpha} t|  \leq |X(x,v,t)| \leq  |x|+(1+\delta) v^{\alpha}t = |x-(1+\delta) v^{\alpha}t| \,.
  \end{equation}
If $x>0$ then 
\begin{align}
 | x- (1-\delta)  v^{\alpha}t | &\leq |X(x,v,t)| \leq |x- (1+\delta) v^{\alpha}t| \qquad  \mbox{ if } t \geq \frac{x}{(1-\delta)v^{\alpha}}\,,\label{xpos1}\\
  |x- (1+\delta) v^{\alpha}t| & \leq |X(x,v,t)| \leq |x- (1-\delta) v^{\alpha}t|  \qquad  \mbox{ if } t \leq \frac{ x}{(1+\delta)v^{\alpha}}, \label{xpos2}
\end{align}
and
\begin{align}
 |X(x,v,t)| \leq 3 v^{\alpha} t\qquad   \mbox{ if } t \in \Big( \frac{x}{(1+\delta)v^{\alpha}}, \frac{ x}{(1-\delta)v^{\alpha}}\Big)\,.\label{xpos3}
\end{align}

 \end{cor}

 \bigskip
\begin{proof}[Proof of Proposition \ref{properties of the characteristics}]
\textbf{Proof of \eqref{Vestimate}:}
 First, we see from \eqref{the original characteristics}   that $X(x,v,t) \leq x$ and $V(x,v,t) \leq v$ for all $t\geq 0$.
Next, we define
\begin{align}\label{definition of h}
    \psi(x):=\int_{-\infty}^{x}\frac{{L}\,\der \xi}{1+|\xi|^{m-d}}
\end{align}
and $\Phi\colon (0,\infty)^2 \to \mathbb{R}$, $(z,v) \to \Phi(z,v)$ via 
\begin{align}\label{definition characteristics solution}
    \Phi(\psi(X(x,v,t))-\psi(x),v):=V(x,v,t).
\end{align}
Then it follows from \eqref{the original characteristics} that 
\begin{equation}\label{Phiequation}
\left\{\begin{aligned}
\partial_z \Phi (\cdot, v)&= \Phi^{\gamma-\alpha} \xi_R(\Phi)\,, \qquad z>0\\
\Phi(0,v)&=v\,.
   \end{aligned}\right.
   \end{equation}
By integrating the ODE in \eqref{Phiequation} we deduce that 
\begin{align}
    \Phi(z,v)&=v& \qquad \mbox{ for all } z\geq 0\,,  v\in[0,R]\,,\label{Phi-1}\\
    \Phi(z,v)&=(v^{1-(\gamma-\alpha)}+(1-(\gamma-\alpha))z)^{\frac{1}{1-(\gamma-\alpha)}}
    & \qquad \mbox{ for all } z\geq 0\,, v\geq 2R\,, \label{Phi-2}\\
    \Phi(z,v)&\leq(v^{1-(\gamma-\alpha)}+(1-(\gamma-\alpha))z)^{\frac{1}{1-(\gamma-\alpha)}}&\qquad \mbox{ for all } z\geq 0
    \,, \, v\geq R\,, \label{Phi-3}\\
            \Phi(z,v)&\geq v & \qquad  \mbox{ for all } z\geq 0\,,\, v\geq R\,. \label{Phi-4}
\end{align}
Notice also that $0\leq \psi(x)\leq C L$, for some constant $C>0$ which is independent of $x$. It thus suffices to consider values of $z$ in the interval $[0,CL]$. Taking in particular $R$ sufficiently large, it follows that $(1-\delta) v\leq \Phi(z,v) \leq (1+\delta)  v$ for all $z \in [0,CL]$ and $v\geq 0$. Due to the definition of $\Phi$ in \eqref{definition characteristics solution}, the estimate \eqref{Vestimate} follows.

\bigskip

\textbf{Proof of \eqref{Xestimate}:} Estimate \eqref{Xestimate} follows then from \eqref{Vestimate} and the relation 
\begin{align}\label{bound x 1}
X(x,v,t)-x=-\int_{0}^{t} V^{\alpha}(x,v,\tau)\der \tau
\end{align}
together with the fact that $(1+\delta)^{\alpha}\leq 1+\delta$ and $(1-\delta)^{\alpha}\geq 1-\delta$.

\bigskip

\textbf{Proof of \eqref{dvXestimate}:} In order to estimate the derivatives with respect to $v$ of the characteristics we 
first prove that 
\begin{align}\label{der in v of phi}
 \frac{1}{2} \leq \partial_{v}\Phi(z,v)\leq 2.
\end{align}
Estimate \eqref{der in v of phi} is immediate if $v \leq R$. If $v \geq R$ we have
 \begin{align*}
  \frac{\der }{\der z}\bigg(\frac{\der \Phi}{\der v}\bigg)=\bigg[(\gamma-\alpha)\Phi^{\gamma-\alpha-1}\xi_{R}(\Phi)+\Phi^{\gamma-\alpha}\xi'_{R}(\Phi)\bigg]\frac{\der \Phi}{\der v}\,, \qquad 
 \frac{\der \Phi}{\der v}(0)=1
  \end{align*}
such that  
\begin{align}\label{derivative in problematic region}
\frac{\der \Phi}{\der v}=\exp \Big( {\int_{0}^{z}\big[(\gamma-\alpha)\Phi^{\gamma-\alpha-1}\xi_{R}(\Phi)+\Phi^{\gamma-\alpha}\xi'_{R}(\Phi)\big]
\der s}\Big) \leq \exp \Big( C R^{\gamma - \alpha -1}\Big)
  \end{align}
and the upper bound in \eqref{der in v of phi} follows since $\gamma <1+\alpha$ 
 and if $R$ is sufficiently large. Analogously, we obtain that
\begin{align*}
    \frac{\der \Phi}{\der v}\geq \exp \Big(- C R^{\gamma - \alpha -1}\Big)
\end{align*}
and \eqref{der in v of phi} follows. In order to proceed, we notice that 
\begin{align}\label{tintegral}
\int_{x}^{X}\frac{1}{\Phi(\psi(\xi)-\psi(x),v)^{\alpha}}\der \xi=-t.
\end{align}

Differentiating in $v$ we obtain
\begin{align}\label{first derivative of x}
&\frac{\partial_{v}X(x,v,t) }{\Phi(\psi(X(x,v,t))-\psi(x),v)^{\alpha}}=  \alpha\int_{x}^{X} \frac{\partial_{v}\Phi(\psi(\xi)-\psi(x),v)}{\Phi(\psi(\xi)-\psi(x),v)^{\alpha+1}}\der \xi.
\end{align}
Since $\delta\in(0,\frac{1}{2})$, \eqref{Vestimate} and \eqref{Xestimate} imply that $\frac{v}{2}\leq V(x,v,t)\leq \frac{3v}{2}$ and that $\frac{v^{\alpha}t}{2}\leq x-X(x,v,t)\leq \frac{3v^{\alpha}t}{2}$. From this we deduce, using \eqref{Xestimate}, the estimates for $\Phi$, and the fact that $\alpha<1$, that
\begin{align}\label{der in v of x}
-\partial_{v}X(x,v,t)=&\alpha\int_{X}^{x} \frac{\partial_{v}\Phi(\psi(\xi)-\psi(x),v)}{\Phi(\psi(\xi)-\psi(x),v)^{\alpha+1}}\der \xi\times  [\Phi(\psi(X(x,v,t))-\psi(x),v)]^{\alpha}\nonumber\\
    &\leq 4\alpha 3^{\alpha} \int_{X}^{x} \frac{1}{v^{\alpha+1}}\der \xi  v^{\alpha} \leq 12\alpha v^{-1}[x-X(x,v,t)]\leq 18\alpha v^{\alpha-1}t.
\end{align}

Analogously we obtain that
\begin{align*}
    -\partial_v X \geq &\frac{\alpha}{2}\int_{X}^{x} \frac{1}{\Phi(\psi(\xi)-\psi(x),v)^{\alpha+1}}\der \xi\times  [\Phi(\psi(X(x,v,t))-\psi(x),v)]^{\alpha}\nonumber\\ 
       &\geq \frac{\alpha}{3^{\alpha+1}} \int_{X}^{x} \frac{1}{v^{\alpha+1}}\der \xi  v^{\alpha} \leq \frac{\alpha}{9} v^{-1}[x-X(x,v,t)]\geq \frac{\alpha}{18} v^{\alpha -1} t,
\end{align*}
which concludes the proof of \eqref{dvXestimate}.

\bigskip

\textbf{Proof of \eqref{dvVestimate}:} We now prove that \eqref{dvVestimate} holds if $x\not\in[(1-2\delta) v^{\alpha}t,  (1+2\delta) v^{\alpha}t]$. From (\ref{definition characteristics solution}) we deduce that 
\begin{align}\label{dervv}
\partial_{v}V(x,v,t)=\partial_{v}\Phi(\psi(X)-\psi(x),v)+\partial_{z}\Phi(\psi(X)-\psi(x),v) \psi'(X)\partial_{v}X.
\end{align}
Due to (\ref{der in v of phi})   it suffices to show that
\begin{align}\label{needed for the appendix}
 |\partial_{z}\Phi(\psi(X)-\psi(x),v) \psi'(X)\partial_{v}X|\leq \frac{1}{4}
 \end{align}
 in order to conclude our proof. From (\ref{definition characteristics solution}), (\ref{Phiequation}) and (\ref{Vestimate}), we have that 
\begin{align*}
    0\leq \partial_{z}\Phi(z,v)\leq 2 v^{\gamma-\alpha}.
\end{align*}
Indeed, by \eqref{definition of h} and \eqref{Vestimate} it holds that
\begin{align*}
    \partial_{z}\Phi(z,v)\leq \Phi(z,v)^{\gamma-\alpha}\leq \max\{2^{\alpha-\gamma},\bigg(\frac{3}{2}\bigg)^{\gamma-\alpha}\}v^{\gamma-\alpha}\leq 2 v^{\gamma-\alpha}.
\end{align*}
By \eqref{dvXestimate}, it thus follows that
\begin{align}\label{negligible term derivative of v}
      |\partial_{z}\Phi(\psi(X)-\psi(x),v) \psi'(X)\partial_{v}X|\leq \frac{36\alpha Lv^{\gamma-\alpha}}{1+|X(t)|^{m-d}}v^{\alpha-1}t=\frac{36\alpha L v^{\gamma-1}t}{1+|X(t)|^{m-d}}.
\end{align}
 We only analyze here the case when $v\geq R$, since by \eqref{Phiequation} we have that $\partial_{z}\Phi(z,v)=0$ when $v\leq R$ and thus there is nothing to prove in this case.
\begin{enumerate}
    \item[1)] Assume $x\leq (1-2\delta) v^{\alpha}t$.
\end{enumerate}
 Then by \eqref{xpos1} we have  $|X(x,v,t)|\geq | x- (1-\delta)  v^{\alpha}t|= (1-\delta) v^{\alpha}t -x \geq \delta  v^{\alpha}t$ and we obtain
 \begin{align*}
     |\partial_{z}\Phi(\psi(X)-\psi(x),v) \psi'(X)\partial_{v}X|\leq C L \frac{ v^{\gamma-1}t}{1+(\delta v^{\alpha}t)^{m-d}}\,.
     \end{align*}
Now, if  $v^{\gamma-1}t\geq \frac{1}{4 L C}$, then $v^{\alpha}t\geq  \frac{1}{4LC}$ since $\alpha>\gamma-1$ and we obtain 
\begin{align*}
     |\partial_{z}\Phi(\psi(X)-\psi(x),v) \psi'(X)\partial_{v}X|&\leq C L \frac{ v^{\gamma-1}t}{1+(\delta v^{\alpha}t)^{m-d}}\,
     \leq \frac{C L}{\delta} v^{\gamma-\alpha-1}(\delta v^{\alpha}t)^{1+d-m}\\
     &\leq C(L,\delta) v^{\gamma-\alpha-1} \leq \frac 1 4
     \end{align*}
     if $R>0$ is sufficiently large.
If $v^{\gamma-1}t\leq  \frac{1}{4LC}$, (\ref{negligible term derivative of v}) becomes
      \begin{align}\label{small mention}
     |\partial_{z}\Phi(\psi(X)-\psi(x),v) \psi'(X)\partial_{v}X|\leq \frac{CLv^{\gamma-1}t}{1+(\delta v^{\alpha}t)^{m-d}}\leq\frac{1}{4}.
     \end{align}
     \begin{enumerate}
    \item[2)] Assume $x\geq (1+2\delta) v^{\alpha}t$.
 Then $|X(x,v,t)| \geq |x- (1+\delta) v^{\alpha}t | = x-(1+\delta) v^{\alpha}t \geq \delta v^{\alpha}t $ and we can conclude as before. 

\end{enumerate}

\textbf{Proof of \eqref{dvVestimate problematic region} and \eqref{dvVestimate decreasing}:} If $x\in[(1-2\delta) v^{\alpha}t,(1+2\delta) v^{\alpha}t]$, then from \eqref{negligible term derivative of v} it follows that
 \begin{align*}
      |\partial_{z}\Phi(\psi(X)-\psi(x),v) \psi'(X)\partial_{v}X|\leq \frac{36Lv^{\gamma-1}t}{1+|X(t)|^{m-d}}\leq 36 L v^{\gamma-1}t
\end{align*}
and \eqref{dvVestimate problematic region} follows.

In order to prove that \eqref{dvVestimate decreasing} holds, we notice that
\begin{align*}
  \partial_{z}\Phi(\psi(X)-\psi(x),v) \psi'(X)\partial_{v}X\leq 0.  
\end{align*}
Combining this with \eqref{dervv} and \eqref{der in v of phi}, the conclusion follows.
\end{proof}



With the help of the characteristics, the solution  $G_{L}$ of (\ref{eq char}) can be written as 
\begin{align}\label{form solution characteristics}
G_{L}(x,v,t)={\frac{C_0}{1+|{X}(x,v,t)|^{m}+{V}(x,v,t)^{p}}}.
\end{align}

Moreover, it holds that
\begin{align}\label{properties derivative g 1}
\partial_{v}G_{L}(x,v,t)=-{ C_0} \frac{[m|X|^{m-2}X\partial_{v}{X}+p{V}^{p-1}\partial_{v}{V}]}{(1+|X|^{m}+{V}^{p})^{2}} \, \big (x,v,t\big).
\end{align}

This function $G_{L}$ will be the main building block for constructing a uniform supersolution to the sequence $\{f_n\}_{n \in \N}$ in Subsection \ref{subsection supersolution}. We would like this supersolution to be decreasing in $v$ for fixed $x$. Unfortunately, for $x>0$ the function $G_{L}$ is not decreasing in $v$. The next proposition characterizes a local maximum of $G_{L}$.

\begin{prop}
$a)$ Given $L>0$ and $\delta\in\big(0,\frac{1}{2}\big)$ there exists a sufficiently large $R>0$ such that for all $t\in[0,1]$ the following holds. For every $x\in\mathbb{R}$ and $t\in[0,1]$ there exists 
at least one point $v_{\max}(x,t)$ with the properties that
\begin{align}\label{first property v max}
v_{\textup{max}}(x,t)^{\alpha}\notin\bigg[\frac{x}{(1+2\delta)t},\frac{x}{(1-2\delta)t}\bigg]
\end{align}
and
\begin{align}\label{definition v max}
 \partial_v G_{L}(x,v_{\max}(x,t),t)=0\,.
\end{align}
Moreover, there exists a constant $K_{max}>0$, which is independent of $x$, $t$, $\delta$, $R$, and $L$, such that the following holds:
\begin{align}\label{vmaxestimate}
 \frac{1}{K_{max}} x t^{\frac{1}{m-1}} \leq v_{\max}(x,t)^{\alpha}
 \leq K_{max} x t^{\frac{1}{m-1}}.
\end{align}
$b)$ Given $L>0$ and $\delta\in\big(0,\frac{1}{2}\big)$ there exists a sufficiently large $R>0$ such that for all $t\in[0,T]$, with $T$ sufficiently small, that is independent of $L,\delta,$ and $R$, there exists a unique point  $v_{\textup{max}}$ with the properties \eqref{first property v max} and \eqref{definition v max}.
\end{prop}

\begin{proof}
 Let $v^{\alpha} \notin \Big [ \frac{x}{(1+2\delta)t}, \frac{x}{(1-2\delta)t}\Big]$ such that (\ref{definition v max}) holds. Notice that from (\ref{properties derivative g 1}), it follows that $m|X|^{m-2}X\partial_{v}{X}+p{V}^{p-1}\partial_{v}{V}=0$. If $R$ is sufficiently large, then the estimates from Proposition \ref{properties of the characteristics} hold. We have the following cases.
 
 \bigskip
 \textbf{Case 1.} $x-(1+2\delta)v^{\alpha}t\geq 0$. Then, by (\ref{Xestimate}), it holds that $0\leq x-(1+\delta) v^{\alpha}t \leq X (x,v,t) \leq x-(1-\delta) v^{\alpha}t$ and by \eqref{dvXestimate}, it follows that $\partial_{v}X\leq 0$. Thus, if \eqref{definition v max} holds, we have that $m|X|^{m-1}|\partial_{v}X|=pV^{p-1}\partial_{v}V$. By Proposition \ref{properties of the characteristics} and since $\delta\in(0,\frac{1}{2})$, we have that there exists a constant $C>0$, that is independent of $\delta, L$ and $R$, such that 
 \begin{align*}
     \frac{1}{C}|x-(1+\delta)v^{\alpha}t|^{m-1}v^{\alpha-1}t\leq m|X|^{m-1}|\partial_{v}X|=pV^{p-1}\partial_{v}V\leq Cv^{p-1}.
 \end{align*}
This implies that $ \frac{1}{C}|x-(1+\delta)v^{\alpha}t|^{m-1}t\leq v^{\alpha(m-1)}.$ Since $m-1$ is odd, $\delta<1$, and $x-(1+2\delta)v^{\alpha}t\geq 0$, we further have that $xt^{\frac{1}{m-1}}\leq Cv^{\alpha}+2v^{\alpha}t^{\frac{m}{m-1}}$. Since $t\leq 1$, we then obtain that  $xt^{\frac{1}{m-1}}\leq Cv^{\alpha}$ and the lower bound in (\ref{vmaxestimate}) follows.
In order to obtain the upper bound in (\ref{vmaxestimate}), we use similar computations together with the fact that $t\geq 0$ and that
  \begin{align*}
C|x-(1-\delta)v^{\alpha}t|^{m-1}v^{\alpha-1}t\geq m|X|^{m-1}|\partial_{v}X|=pV^{p-1}\partial_{v}V\geq \frac{1}{C}v^{p-1}.
 \end{align*}

 \bigskip
 \textbf{Case 2.} $x-(1-2\delta)v^{\alpha}t\leq 0$. Then, by (\ref{Xestimate}), it holds that $ x-(1+\delta) v^{\alpha}t \leq X (x,v,t) \leq x-(1-\delta) v^{\alpha}t\leq 0$ and thus $m|X|^{m-2}X\partial_{v}{X}+p{V}^{p-1}\partial_{v}{V}=m|X|^{m-1}|\partial_{v}{X}|+p{V}^{p-1}\partial_{v}{V}>0$ in this region. Using this and \eqref{properties derivative g 1}, it follows that $\partial_{v}G_{L}(x,v,t)<0$.

In Appendix \ref{appendix c}, Proposition \ref{0 prop second derivative}, we will prove that for a sufficiently small $T>0$, which is independent of $L,\delta,$ and $R$, we have that
\begin{align}\label{second derivative is monotone}
\partial_{v}^{2}G_{L}(x,v,t)< 0,
\end{align}
for all $v$ that satisfy the estimate in \eqref{vmaxestimate} and all $t\in[0,T]$ which implies the uniqueness of such a point. This concludes our proof.
\end{proof}

The following lemma will also be needed in the construction of a supersolution in Section 
\ref{subsection supersolution}.

\begin{prop}\label{derivatives in x variable}
 Given $L>0$ and $\delta\in\big(0,\frac{1}{2}\big)$ there exists a sufficiently large $R>0$ such that for all $t\in[0,T]$, with $T$ sufficiently small, which is independent of $L,\delta,$ and $R$, it holds that
\begin{align}
\partial_{x}G_{L}(x,v_{\textup{max}}(x,t),t)\leq 0,\label{upper lower bound for supersolution derivatives in x}
\end{align}
where $G_{L}$ is the solution of (\ref{eq char}) and $v_{\textup{max}}(x,t)$ was defined in (\ref{definition v max}).
\end{prop}
\begin{proof}
If $R$ is sufficiently large, then the estimates from Proposition \ref{properties of the characteristics} hold. We first notice that due  to \eqref{Xestimate} and \eqref{vmaxestimate} we have
\[
 X(x,v_{\max}(x,t),t) \geq x - 2 v_{\max}(x,t)^{\alpha} t \geq x - 2K_{max}x t^{1+\frac{1}{m-1}} >0
\]
if $t\in [0,T]$ and $T$ is sufficiently small. Then we compute
\begin{align*}
    \partial_{x}G_{L}(x,v,t)=-C_{0}\frac{[m|X|^{m-2}X\partial_{x}{X}+p{V}^{p-1}\partial_{x}{V}]}{(1+|X|^{m}+{V}^{p})^{2}}(x,v,t)\,,
\end{align*}
and thus it suffices to prove that 
\begin{align}\label{what we need to prove}
    \Big(m X^{m-1}\partial_{x}{X}+p{V}^{p-1}\partial_{x}{V} \Big)(x,v_{\textup{max}}(x,t),t)\geq 0\,.
\end{align}
 
Since $V \geq 0$ we  have to prove that $\partial_{x}{X}(x,v_{\textup{max}}(x,t),t)\geq 0$ and $\partial_{x}{V}(x,v_{\textup{max}}(x,t),t)\geq 0.$ We start by analyzing $\partial_{x}{X}$.

We consider the case when $v\geq 2R$, since the other cases work similarly.  Differentiating  \eqref{tintegral} with respect to $x$, keeping in mind that $x>0$, we obtain
\begin{align}
\frac{\partial_{x}X(x,v,t)}{\Phi(\psi(X(x,v,t))-\psi(x),v)^{\alpha}}  +  \alpha\int_{x}^{X} \frac{\partial_{z}\Phi \partial_{x}\psi(x)\der \xi}{\Phi(\psi(\xi)-\psi(x),v)^{\alpha+1}}-\frac{1}{\Phi(0,v)^{\alpha}}=0.
\end{align}
We have that $\partial_z \Phi \geq 0$ due to \eqref{Phiequation}. Since $X(x,v,t)\leq x$ and by \eqref{definition characteristics solution} and \eqref{Vestimate} we obtain
\[
 \partial_x X(x,v,t) \geq \frac{\Phi(\psi(X(x,v,t))-\psi(x),v)^{\alpha}}{v^{\alpha}} \geq \frac{1}{C}>0\,.
\]

Thus, in order for (\ref{what we need to prove}) to hold, what is left to prove is that
 $   \partial_{x}{V}(x,v_{\textup{max}}(x,t),t)\geq 0.$
Integrating (\ref{the original characteristics}) over time and differentiating with respect to $x$, we obtain    \begin{align}\label{derivative x of v positive}
     {V}(x,v,t)^{-\gamma}\partial_{x} {V}(x,v,t)&=\int_{0}^{t}\frac{{L}(m-d)|X|^{m-d-2}X\partial_{x}{X}\der\xi}{(1+{X}(\xi)^{m-d})^{2}}.
  \end{align}
  Since we have ${X}(x,v_{\max}(x,t),t)\geq 0$ and $\partial_{x}{X}(x,v,t)\geq 0$ we obtain $\partial_{x}{V}(x,v_{\textup{max}}(x,t),t)\geq 0.$ 
\end{proof}

\subsection{Construction of a supersolution for a linear reformulation of the model}\label{subsection supersolution}
In this subsection we will prove that the solution $G_{L}$ of \eqref{eq char}, after suitable modifications, is a supersolution for  problem (\ref{inductive sequence}).
\begin{defi}\label{definition h for all l}
Let $\delta\in\big(0, \frac{1}{2}\big)$, $L>0$ and $R$, which depends on $L$ and $\delta$, as in Proposition \ref{properties of the characteristics}. Let $G_{L}$ be the solution of \eqref{eq char} with given $\delta, L,$ and $ R$. We define the function $H_{L}$ via
    \begin{align*}
    H_{L}(x,v,t)=\left\{
\begin{array}{ll}
G_{L}(x,v,t), & \mbox{ if } \partial_{v}G_{L}(x,v,t)\leq 0 \; \mbox{ or } \; x \in [(1-2\delta) v^{\alpha}t, (1+2\delta) v^{\alpha}t] ;\\
G_{L}( x,v_{\textup{max}}(x,t),t), &  \textup{otherwise}, \\
\end{array}
\right.
\end{align*}
where $v_{\textup{max}}(x,t)$ was defined in (\ref{definition v max}). 
\end{defi}
 By the choice of $v_{\textup{max}}$ in (\ref{definition v max}), if $x\notin [(1-2\delta) v^{\alpha}t, (1+2\delta) v^{\alpha}t]$ and $v\leq v_{\textup{max}}(x,t)$, then we have that $H_{L}(x,v,t)=G_{L}(x,v_{\textup{max}}(x,t),t)$. Moreover, $H_{L}$ is decreasing in $v$ for fixed $x$ outside possibly a critical region where 
$x \in [(1-2\delta) v^{\alpha}t, (1+2\delta) v^{\alpha}t]$.

We first collect some properties of the function $H_{L}$, which are independent of $\delta, L,$ and $ R$. 
\begin{lem}\label{L.Hproperties}
Let $\delta\in\big(0,\frac{1}{2}\big)$, $L>0$,  and $R$ as in Proposition \ref{properties of the characteristics}. Let $T>0$ be sufficiently small, independent of $\delta,$ $L$, and $R$, and  $   H_{L}$ as in Definition \ref{definition h for all l}. Then there exists a constant $K_{2}>0$, which is independent of $\delta,$ $L$, and $R$, such that the following holds for $t \in [0,T]$.

If $x>0$ then 
\begin{align}\label{new H estimate}
    H_{L}(x,v-v',t)\leq K_{2} H_{L}(x,v,t) \textup{ for all } v'\in\bigg(0,\frac{v}{2}\bigg)
\end{align}

If $x>0$ and for all $v$ such that $v^{\alpha} \notin \big[ \frac{x}{(1+2\delta)t}, \frac{x}{(1-2\delta)t}\big]$ then there exists a sufficiently large constant $C_{l}>0$ such that
\begin{align}
-\partial_{v}H_{L}(x,v',t)  \leq -K_{2} \partial_v H_{L}(x,v,t) \qquad \mbox{ for all } v \geq \max\{R,C_{l}v_{\max}(x,t)\}\,, v'\in \Big( \tfrac{v}{2},v\Big).\label{H3}
 \end{align}

If $x\leq 0$ then 
\begin{align}\label{H4}
 -\partial_v H_{L}(x,v',t) \leq - K_{2} \partial_v H_{L}(x,v,t) \qquad \mbox{ for all } v \geq R \,, v'\in \big( \tfrac v 2 ,v\big)\,. 
\end{align}
\end{lem}

\begin{proof}
{\bf Proof of \eqref{new H estimate}:}  We will prove that (\ref{new H estimate}) holds by proving separately that there exists $K_{2}>0$, which is independent of $\delta,L,$ and $R$, such that:
 
 If $x>0$ and for all $v$ such that $v^{\alpha} \notin \big[ \frac{x}{(1+2\delta)t}, \frac{x}{(1-2\delta)t}\big]$ then 
  \begin{align}
H_{L}(x,v-v',t) & \leq K_{2} H_{L}(x,v,t) \qquad \mbox{ for all } v'\in \Big (0,\tfrac{v}{2}\Big) \mbox{ and } v,v-v'\geq v_{\textup{max}}(x,t)\,, \label{H1}\\
H_{L}(x,v_{\max}(x,t),t) & \leq K_{2} H_{L}(x,v,t)\qquad \mbox{ for all } v \in (v_{\max}(x,t), 2 v_{\max}(x,t)),\label{H2}
\end{align}
 If $x>0$ and for all $v$ such that $v^{\alpha} \in \big[ \frac{x}{(1+2\delta)t}, \frac{x}{(1-2\delta)t}\big]$ then
 \begin{align}\label{H estimate problematic region}
     H_{L}(x,v-v',t)\leq K_{2} H_{L}(x,v,t) \textup{ for all }  v'\in\bigg(0,\frac{v}{2}\bigg).
 \end{align}

Before beginning our proof, we make the following observation. Because the proof of each region when $v^{\alpha} \notin \big[ \frac{x}{(1+2\delta)t}, \frac{x}{(1-2\delta)t}\big]$ differs, we have to distinguish between different cases. 

If $v^{\alpha}\geq  \frac{x}{(1-2\delta)t}$ and  $w\in[\frac{v}{2},v]$, we have the subcases
\begin{itemize}
    \item[1. a)] $x>0$,  $v^{\alpha}\geq  \frac{x}{(1-2\delta)t}$, and $w^{\alpha}\geq \frac{x}{(1-2\delta)t}$. Notice that in this region $x-(1-\delta)w^{\alpha}t\leq 0$;
    \item[1. b)] $x>0$, $v^{\alpha}\geq  \frac{x}{(1-2\delta)t}$, and $w^{\alpha}\in\bigg[\frac{x}{(1+2\delta)t}, \frac{x}{(1-2\delta)t}\bigg]$;
    \item[1. c)] $x>0$, $v^{\alpha}\geq  \frac{x}{(1-2\delta)t}$, and $w^{\alpha}\leq  \frac{x}{(1+2\delta)t}$.  Notice that in this region $x-(1+\delta)w^{\alpha}t\geq  0$.
\end{itemize}
The remaining case is
\begin{itemize}
    \item[2.] $x>0$, $v^{\alpha}\leq \frac{x}{(1+2\delta)t}$ and then $w^{\alpha}\leq \frac{x}{(1+2\delta)t}$.
\end{itemize}

\bigskip
{\bf Proof of \eqref{H1}:} 
We will prove that there exists a constant $C>0$, which is independent of $\delta,L,$ and $R$, such that
\begin{align}\label{main inequality for v-v' h}
\frac{    1}{C(1+|x|^{m}+v^{p})}  \leq   H_{L}(x,w,t)\leq 
  \frac{ C}{1+|x|^{m}+v^{p}},
\end{align}
for all $w\in[\frac{v}{2},v]$. 

\bigskip
\textbf{Case $1.$ $a)$} Notice that because of (\ref{Xestimate}), we are in the region where $X(x,w,t)\leq 0$. Because of (\ref{xpos1}), it follows that $|x-(1-\delta)w^{\alpha}t|\leq |X(x,w,t)|\leq |x-(1+\delta)w^{\alpha}t|$. Thus, due to (\ref{form solution characteristics}), it suffices to show in this case that 
\begin{align}\label{Jestimate}
 |x-(1+\delta)w^{\alpha}t|^m+w^p \leq C\big( |x|^m + v^p\big) \leq C\bigg(  |x-(1-\delta)w^{\alpha}t|^m+w^p \bigg).
\end{align}
To prove \eqref{Jestimate} we notice that is suffices to show that $|x-(1+\delta)w^{\alpha}t|^m+w^p \leq C\big( |x|^m + w^p\big) \leq C\big(  |x-(1-\delta)w^{\alpha}t|^m+w^p \big)$ since $w \in [\frac{v}{2},v]$. Since $x-(1-\delta)w^{\alpha}t\leq 0$ in this case, it holds that
$|x-(1-\delta)w^{\alpha}t|\leq w^{\alpha}$. In this case, we have that
$|x|-(1-\delta)w^{\alpha}t\leq  |x-(1-\delta)w^{\alpha}t|\leq  w^{\alpha}$.
This implies that $|x|\leq 2 w^{\alpha}$. Notice that since $x\geq 0$ and $x-(1-\delta)w^{\alpha}t\leq 0$, it immediately follows that $0\leq x \leq w^{\alpha}$. However, we do a more general proof as a similar estimate will be needed in order to prove (\ref{H4}) or when $x-(1+\delta)w^{\alpha}\geq 0$ later on.

Since $(|x|+(1+\delta)w^{\alpha}t)^{m}\leq C_{m}(|x|^{m}+t^{m}w^{p})$ and  $t$ is sufficiently small, we obtain
\begin{align}\label{ineq}
|x-(1+\delta)w^{\alpha} t|^m + w^p &\leq (|x|+(1+\delta)w^{\alpha}t)^{m}+w^{p}\leq C_{m}(|x|^{m}+t^{m}w^{p})+w^{p}\nonumber\\
&\leq C(|x|^{m}+w^{p}).
\end{align}
Additionally, since $|x|\leq 2 w^{\alpha}$, it holds that
\begin{align}\label{ineq two}
|x-(1-\delta)w^{\alpha}t|^m + w^p \geq w^{p}=\frac{w^{p}}{2}+\frac{w^{p}}{2}\geq  \frac{1}{C}(|x|^{m}+w^{p})\,.
\end{align}

\bigskip
\textbf{Case $1.$ $b)$}
From (\ref{Xestimate}) and (\ref{xpos3}), it follows that $X(x,v,t)\leq 0$ and $|X(x,w,t)|\leq C w^{\alpha}t$. Since in this case $v^{\alpha}\geq  \frac{x}{(1-2\delta)t}$, from \eqref{Jestimate} we know that there exists a constant $C>0$, which is independent of $\delta,L,$ and $R$, such that
\begin{align*}
    \frac{ 1}{C(1+|x|^{m}+v^{p})}  \leq   H_{L}(x,v,t).
\end{align*}
Thus, in order for (\ref{H1}) to hold in this case, we need to prove that there exists a constant $C>0$, which is independent of $\delta,L,$ and $R$, such that  $H_{L}(x,w,t)\leq \frac{ C}{1+|x|^{m}+v^{p}} $, for $w\in[\frac{v}{2},v]$. More precisely, due to (\ref{form solution characteristics}), it suffices to prove that
\begin{align}\label{case 1 b}
    1+|X(x,w,t)|^{m}+V(x,w,t)^{p}\geq \frac{1}{C}(1+|x|^{m}+v^{p}).
\end{align}

Since $t\leq 1$, we have that $x\leq (1+2\delta)w^{\alpha}t\leq 2w^{\alpha}$ in this case. Due to \eqref{Vestimate}, it holds that
\begin{align*}
    |X(x,w,t)|^{m}+V(x,w,t)^{p}\geq V(x,w,t)^{p}\geq Cw^{p}=\frac{ Cw^{p}}{2}+\frac{ Cw^{p}}{2}\geq C(|x|^{m}+w^{p}).
\end{align*}

\bigskip
\textbf{Case $1.$ $c)$} From (\ref{Xestimate}), it follows that $X(x,v,t)\leq 0$ and $X(x,w,t)\geq 0$. Since in this case $v^{\alpha}\geq  \frac{x}{(1-2\delta)t}$, from \eqref{Jestimate} we know that there exists a constant $C>0$, independent of $\delta,L,$ and $R$, such that $\frac{ 1}{C(1+|x|^{m}+v^{p})}  \leq   H_{L}(x,v,t)$ and we need to prove that $  H_{L}(x,w,t)\leq \frac{ C}{1+|x|^{m}+v^{p}} $, for $w\in[\frac{v}{2},v]$. More precisely, due to (\ref{xpos2}) and (\ref{form solution characteristics}), it suffices to prove as  before that    $1+|X(x,w,t)|^{m}+V(x,w,t)^{p}\geq C(1+|x|^{m}+v^{p})$. Actually, we prove here a more general estimate that will be used in Case $2.$, namely
\begin{align}\label{positive part one}
 |x-(1-\delta)w^{\alpha}t|^m+w^p \leq C\big( |x|^m + v^p\big) \leq C\bigg(  |x-(1+\delta)w^{\alpha}t|^m+w^p \bigg)
\end{align}

In order to prove \eqref{positive part one}, we distinguish between two cases:
\begin{enumerate}
     \item [i)]
 $|x-(1+\delta)w^{\alpha}t|\geq w^{\alpha}$. In this case, we have that
$w^{\alpha}\leq |x-(1+\delta)w^{\alpha}t|\leq |x|+(1+\delta)w^{\alpha}t$, which  implies that $\frac{1}{2}w^{\alpha}\leq |x|$. 
Thus $|x|-(1+\delta)w^{\alpha}t\geq |x|-\frac{w^{\alpha}}{4}\geq \frac{|x|}{2}>0$ and $(|x|-(1+\delta)w^{\alpha}t)^{m}\geq \frac{|x|^{m}}{2^{m}}$. Since $|x|-(1+\delta)w^{\alpha}t>0$ and remembering that $m$ is even, we also have that $|x-(1+\delta)w^{\alpha}t|^{m}\geq (|x|-(1+\delta)w^{\alpha}t)^{m}$ and thus $|x-(1+\delta)w^{\alpha}t|^{m}\geq 2^{-m}|x|^{m}$.
Additionally , since $t$ is sufficiently small and  $\frac{1}{2}w^{\alpha}\leq |x|$, it holds that $|x-(1-\delta)w^{\alpha}t|\leq |x|+w^{\alpha}t\leq 2|x|$.
\item[ii)]  $|x-(1+\delta)w^{\alpha}t|\leq w^{\alpha}$. This case can be treated as in the proof of \eqref{Jestimate}.
\end{enumerate}

\bigskip
\textbf{Case $2.$}  From (\ref{Xestimate}), it follows that $X(x,v,t)\geq 0$ and $X(x,w,t)\geq 0$. Thus, from (\ref{xpos2}) and (\ref{positive part one}), the conclusion follows.

\bigskip
{\bf Proof of \eqref{H2}:} 
Notice that due to (\ref{vmaxestimate}) and since $v\in(v_{\textup{max}}(x,t),2v_{\textup{max}}(x,t))$, we have that $x-(1+2\delta)v^{\alpha}t\geq 0$ if we choose $t$ to be sufficiently small. Moreover, since $\delta<\frac{1}{2}$, $t$ can be chosen independently of $\delta$. Since $v>v_{\textup{max}}(x,t)$, it holds by (\ref{positive part one}) that $\frac{    1}{C(1+|x|^{m}+v^{p})}  \leq   H_{L}(x,v,t)$. Thus, since  $v\in(v_{\textup{max}}(x,t),2v_{\textup{max}}(x,t))$, we further have that  $\frac{    1}{C(1+|x|^{m}+v_{\textup{max}}(x,t)^{p})}  \leq   H_{L}(x,v,t)$. From (\ref{form solution characteristics}) and (\ref{xpos2}), it suffices to prove that there exists a constant $C>1$, independent of $\delta,L,$ and $R$, such that
\begin{align*}
   |x|^{m}+v_{\textup{max}}(x,t)^{p}\leq C\bigg( |x-(1+\delta)v_{\textup{max}}(x,t)^{\alpha}t|^{m}+v_{\textup{max}}(x,t)^{p}\bigg).
\end{align*}
The inequality holds since due to (\ref{vmaxestimate}) we can choose $t$ sufficiently small such that $|x-(1+\delta)v_{\textup{max}}(x,t)^{\alpha}t|=x-(1+\delta)v_{\textup{max}}(x,t)^{\alpha}t\geq \frac{x}{2}$.

\bigskip
{\bf Proof of \eqref{H estimate problematic region}:}
Due to  (\ref{form solution characteristics}) and since $v-v'\in(\frac{v}{2},v)$ it holds on one side that
\begin{align*}
 H_{L}(x,v-v',t) \leq 
 \frac{C}{1  + (v-v')^p}\leq \frac{C}{1+v^{p}},
\end{align*}
for some $C>0$, independent of $\delta,L,$ and $R$. On the other side, similarly with (\ref{xpos3}) and from (\ref{form solution characteristics}), we have that
\begin{align*}
H_{L}(x,v,t)\geq  \frac{C}{1+ v^{p}t^m+v^p}  \geq \frac{C}{1+v^p},
\end{align*}
for some $C>0$, independent of $\delta,L,$ and $R$. Combining the two inequalities, we can conclude that  \eqref{H estimate problematic region} holds.

\bigskip
{\bf Proof of \eqref{H3}:}

\bigskip

   \textbf{Case $1.$ $a)$} We will prove that in this case it holds that
\begin{align}\label{main inequality for the derivative of hh}
 \frac{1}{C}\frac{    v^{p-1}}{(1+|x|^{m}+v^{p})^{2}}  \leq  -\partial_{v} H_{L}(x,w,t)\leq 
 C \frac{ v^{p-1}}{(1+|x|^{m}+v^{p})^{2}},
\end{align}
for some $C>0$, independent of $\delta,L,$ and $R$, and for all $w\in[\frac{v}{2},v]$. We remember we are in the case when $x-(1+\delta)w^{\alpha}t\leq x-(1-\delta)w^{\alpha}t\leq 0$ and thus due to (\ref{Xestimate}) it holds that $|X|^{m-2}(x,w,t)X(x,w,t)\partial_{v}{X}(x,w,t)(x,w,t)=|X|^{m-1}|\partial_{v}X(x,w,t)|$. Due to (\ref{Vestimate})-(\ref{dvVestimate}) we thus have in this region that 
\begin{align*}
   -\frac{1}{C}(x-(1-\delta)w^{\alpha}t)^{m-1}w^{\alpha-1}t+\frac{1}{C}(1-\delta)w^{p-1}&\leq |X|^{m-2}X\partial_{v}{X}(x,w,t)+{V}^{p-1}\partial_{v}{V}(x,w,t)\\
   &\leq -C(x-(1+\delta)w^{\alpha}t)^{m-1}w^{\alpha-1}t+C(1+\delta)w^{p-1}. 
\end{align*} 
Moreover, from (\ref{Jestimate}), we have that
\begin{align}
 |x-(1+\delta)w^{\alpha}t|^m+w^p \leq C\big( |x|^m + v^p\big) \leq C\bigg(  |x-(1-\delta)w^{\alpha}t|^m+w^p \bigg).
\end{align}
In order to prove (\ref{main inequality for the derivative of hh}), due to \eqref{properties derivative g 1}, it then suffices to show in this case that 
\begin{align}\label{Jderestimate}
 |x-(1+\delta)w^{\alpha}t|^{m-1} w^{\alpha-1}t + w^{p-1} \leq C v^{p-1}\leq C\bigg(|x-(1-\delta)w^{\alpha}t|^{m-1} w^{\alpha-1}t + w^{p-1}\bigg),
\end{align}
for all $w \in [\frac{v}{2},v]$. We thus prove \eqref{Jderestimate}. We have  
\begin{align}\label{logic to get rid of w alpha minus one}
    w^{\alpha-1}t|x-(1+\delta)w^{\alpha}t|^{m-1}+w^{p-1}=w^{\alpha-1}\bigg(|x-(1+\delta)w^{\alpha}t|^{m-1} t+w^{\alpha(m-1)}\bigg).
\end{align}
Since $w\in[\frac{v}{2},v]$, it suffices to prove  
\begin{align}\label{tildeJestimate}
\tilde J:= |x-(1+\delta)w^{\alpha}t|^{m-1}t + w^{\alpha(m-1)}\leq C  w^{\alpha(m-1)}
\end{align}
and
\begin{align*}
  |x-(1-\delta)w^{\alpha}t|^{m-1}t + w^{\alpha(m-1)}\geq  \frac{1}{C}  w^{\alpha(m-1)}.
\end{align*}

We are in the case $v\geq C_{l}v_{\textup{max}}(x,t)$, for some  sufficiently large $C_{l}>0$ and $x>0$. We know that $v_{\textup{max}}(x,t)^{\alpha}\geq \frac{1}{K_{max}} xt^{\frac{1}{m-1}}$, for $K_{max}$ as in \eqref{vmaxestimate}, and thus $w^{\alpha}\geq \frac{1}{C}xt^{\frac{1}{m-1}}\geq 0$ since $w\geq \frac{v}{2}$. 

We have that $ |x-(1-\delta)w^{\alpha}t|^{m-1}t + w^{\alpha(m-1)} \geq w^{\alpha(m-1)}$.
On the other hand, since $0\leq xt^{\frac{1}{m-1}}\leq C  w^{\alpha}$ and $t\leq 1$, it follows that
\begin{align*}
     \tilde{J}(x,w,t)\leq C_{m}(|x|^{m-1}t+w^{\alpha(m-1)}t^{m})+w^{\alpha(m-1)}\leq Cw^{\alpha(m-1)}.
\end{align*}
\begin{rmk}\label{rmk all cases}
    Since $v^{\alpha}\geq \frac{x}{(1-2\delta)t}$, we know from \eqref{main inequality for the derivative of hh} that $\frac{1}{C}\frac{    v^{p-1}}{(1+|x|^{m}+v^{p})^{2}}  \leq  -\partial_{v} H_{L}(x,v,t).$ Thus, for Case $1.$ $b)$ and Case $1.$ $c)$ we need to prove that $ -\partial_{v} H_{L}(x,w,t)\leq \frac{   C v^{p-1}}{(1+|x|^{m}+v^{p})^{2}}$, for some $C>0$, which is independent of $\delta,L,$ and $R$.
\end{rmk}

\bigskip

\textbf{Case $1.$ $b)$} As mentioned above, we need to prove that $ -\partial_{v} H_{L}(x,w,t)\leq \frac{   C v^{p-1}}{(1+|x|^{m}+v^{p})^{2}}$ when $w^{\alpha}\in\bigg[\frac{x}{(1+2\delta)t}, \frac{x}{(1-2\delta)t}\bigg]$, $w\in \big[\frac{v}{2},v\big]$. More precisely, due to (\ref{properties derivative g 1}), it suffices to prove that
\begin{align*}
    1+|X(x,w,t)|^{m}+V(x,w,t)^{p}\geq C(1+|x|^{m}+v^{p})
\end{align*}
and that
\begin{align*}
    X(x,w,t)^{m-1}\partial_{v}X(x,w,t)+V(x,w,t)^{p-1}\partial_{v}V(x,w,t)\leq C v^{p-1},
\end{align*}
for all $w\in[\frac{v}{2},v]$ such that $w^{\alpha}\in\big[  \frac{x}{(1+2\delta)t},\frac{x}{(1-2\delta)t}\big]$.

 We know from (\ref{case 1 b}) that the first inequality holds and thus we focus on proving the second inequality. Due to \eqref{Vestimate} and \eqref{dvVestimate decreasing}, it follows that
\begin{align*}
   X(x,w,t)^{m-1}\partial_{v}X(x,w,t)+V(x,w,t)^{p-1}\partial_{v}V(x,w,t) \leq  X(x,w,t)^{m-1}\partial_{v}X(x,w,t)+Cw^{p-1},
\end{align*}
where $C>0$ is independent of $\delta,L,$ and $ R$. We now analyze the term $X(x,w,t)^{m-1}\partial_{v}X(x,w,t)$. From \eqref{xpos3} and \eqref{dvXestimate}, it holds that
\begin{align*}
     X(x,w,t)^{m-1}\partial_{v}X(x,w,t)\leq  |X(x,w,t)|^{m-1}|\partial_{v}X(x,w,t)|\leq Cw^{\alpha(m-1)}w^{\alpha-1}t^{m}\leq Cw^{p-1},
\end{align*}
for some $C>0$, which is independent of $\delta,L,$ and $ R$, and we can conclude by using that $w\in\big[\frac{v}{2},v\big]$.

\bigskip 

 \textbf{Case $1.$ $c)$}   As before, by Remark \ref{rmk all cases}, we only need to prove that    $1+|X(x,w,t)|^{m}+V(x,w,t)^{p}\geq C(1+|x|^{m}+v^{p})$ and that
    $X(x,w,t)^{m-1}\partial_{v}X(x,w,t)+V(x,w,t)^{p-1}\partial_{v}V(x,w,t)\leq C v^{p-1}$ when $w^{\alpha}\leq \frac{x}{(1+2\delta)t}$, $w\in \big[\frac{v}{2},v\big]$. 

By \eqref{Xestimate} and \eqref{dvXestimate} it holds that $X(x,w,t)\geq 0$ and $\partial_{v}X(x,w,t)\leq 0$ in this region. Moreover, due to \eqref{Vestimate} and \eqref{dvVestimate}, it follows that
\begin{align*}
    X(x,w,t)^{m-1}\partial_{v}X(x,w,t)+V(x,w,t)^{p-1}\partial_{v}V(x,w,t)\leq V(x,w,t)^{p-1}\partial_{v}V(x,w,t)\leq C w^{p-1} .
\end{align*}

The fact that $1+|X(x,w,t)|^{m}+V(x,w,t)^{p}\geq C(1+|x|^{m}+v^{p})$ in this region follows from \eqref{positive part one}.

\bigskip

\textbf{Case $2.$} We will prove that \eqref{main inequality for the derivative of hh} holds in this case too. Due to (\ref{Xestimate}), we have that $X(x,v,t),$ $X(x,w,t)\geq 0$ in this region. Thus, from (\ref{xpos2}), it holds that $|x-(1+\delta)w^{\alpha}t|\leq |X(x,w,t)|\leq |x-(1-\delta)w^{\alpha}t|.$ We know from (\ref{positive part one}) that 
$|x-(1-\delta)w^{\alpha}t|^m+w^p \leq C\big( |x|^m + v^p\big) \leq C\bigg(  |x-(1+\delta)w^{\alpha}t|^m+w^p \bigg)$, for some constant $C>0$, which is independent of $\delta,L,$ and $R$. Moreover, since $X(x,w,t)\geq 0$ and $\partial_{w}X(x,w,t)\leq 0$, we have that $X(x,w,t)\partial_{w}X(x,w,t)=-|X(x,w,t)| |\partial_{w}X(x,w,t)|.$ Thus, due to  \eqref{properties derivative g 1}, it suffices to show in this case that there exists $C>0$, independent of $\delta,L,$ and $R$ such that
\begin{align}\label{positive part two}
 -|x-(1+\delta)w^{\alpha}t|^{m-1} w^{\alpha-1}t + w^{p-1} \leq C v^{p-1}\leq C\bigg(-|x-(1-\delta)w^{\alpha}t|^{m-1} w^{\alpha-1}t + w^{p-1}\bigg),
\end{align}
for all $w \in [\frac{v}{2},v]$.

We remember we are in the case when $x-(1+\delta)w^{\alpha}t\geq 0$. It is clear that $-|x-(1+\delta)w^{\alpha}t|^{m-1} w^{\alpha-1}t + w^{p-1} \leq C v^{p-1}$. For the other inequality, due to \eqref{logic to get rid of w alpha minus one} it suffices to prove the statement for $-|x-(1-\delta)w^{\alpha}t|^{m-1} t + w^{\alpha(m-1)} $. Since $x-(1-\delta)w^{\alpha}t\geq 0$ and using that $a^{m-1}+b^{m-1}\leq (a+b)^{m-1}$, for $a,b\geq 0$, we have that $x^{m-1}-((1-\delta)w^{\alpha}t)^{m-1}\geq (x-(1-\delta)w^{\alpha}t)^{m-1}$. Thus, it holds that
\begin{align}\label{ineqqqqqq}
    -|x-(1-\delta)w^{\alpha}t|^{m-1} t + w^{\alpha(m-1)} \geq -|x|^{m-1}t+(1-\delta)^{m-1}w^{\alpha(m-1)}t^{m}+w^{\alpha(m-1)}.
\end{align}
Since $w\geq \frac{v}{2}\geq \frac{C_{l}}{2}v_{\textup{max}}(x,t)$ in this case we have that $xt^{\frac{1}{m-1}}\leq \frac{2^{\alpha}K_{max}w^{\alpha}}{C_{l}^{\alpha}}$, for a sufficiently large constant $C_{l}>0$. (\ref{ineqqqqqq}) thus becomes
\begin{align*}
    -|x-(1-\delta)w^{\alpha}t|^{m-1} t + w^{\alpha(m-1)} \geq -\frac{2^{p-\alpha}K_{max}^{m-1}w^{\alpha(m-1)}}{C_{l}^{p-\alpha}}+w^{\alpha(m-1)}\geq \frac{w^{\alpha(m-1)}}{2},
\end{align*}
for $C_{l}$ sufficiently large, thus concluding our proof.

\bigskip

\bigskip
{\bf Proof of \eqref{H4}:}
In order to prove \eqref{H4}, it is useful to notice that if $x\leq 0$, then \eqref{xnegative} holds. We will prove that there exists $C>0$, which is independent of $\delta,L,$ and $R$, such that
\begin{align}\label{h estimate x negative}
  \frac{1}{C}\frac{ v^{\alpha-1}|x|^{m-1}t+v^{p-1}}{(1+|x|^{m}+v^{p})^{2}} \leq - \partial_{v} H_{L}(x,w,t)\leq C\frac{ v^{\alpha-1}|x|^{m-1}t+v^{p-1}}{(1+|x|^{m}+v^{p})^{2}},
\end{align}
for all $w\in[\frac{v}{2},v]$. Using similar computations as the ones for \eqref{positive part one}, we can prove that
\begin{align}\label{x neg below}
 (|x|+(1+\delta)w^{\alpha}t)^m+w^p \leq C\big( |x|^m + v^p\big) \leq C\bigg(  (|x|+(1-\delta)w^{\alpha}t)^m+w^p \bigg).
\end{align}
 Since $X(x,w,t)\leq 0$ and $\partial_{w}X(x,w,t)\leq 0$ we have that $X(x,w,t)\partial_{w}X(x,w,t)=|X(x,w,t)|$ $|\partial_{w}X(x,w,t)|$. Due to \eqref{xnegative} and \eqref{properties derivative g 1}, what is left to prove in order for (\ref{h estimate x negative}) to hold is that
\begin{align}\label{x neg above}
 (|x|+(1+\delta)w^{\alpha}t)^{m-1} w^{\alpha-1}t + w^{p-1}& \leq C  (v^{\alpha-1}|x|^{m-1}t+v^{p-1})\nonumber\\
 &\leq C\bigg((|x|+(1-\delta)w^{\alpha}t)^{m-1} w^{\alpha-1}t + w^{p-1}\bigg),
\end{align}
for some $C>0$, independent of $\delta,L,$ and $R$, for all $w \in [\frac{v}{2},v]$.

Thus, we only need to prove \eqref{x neg above}. For $x\leq 0$ and $w\in [\frac{v}{2},v]$, we have that $  (|x|+(1-\delta)w^{\alpha}t)^{m-1}t+w^{\alpha(m-1)}\geq |x|^{m-1}t + w^{\alpha(m-1)}$.
Furthermore, since  $t\leq 1$, it follows that 
\begin{align*}
    (|x|+(1+\delta)w^{\alpha}t)^{m-1}t+w^{\alpha(m-1)}&\leq C_{m}|x|^{m-1}t+C_{m}w^{\alpha(m-1)}t^{m}+w^{\alpha(m-1)}\\
    &\leq C(|x|^{m-1}t+w^{\alpha(m-1)}).
\end{align*}

\end{proof}
 We also prove some moment bounds for the function $H_{L}$, which are independent of $\delta, L,$ and $R$. 
\begin{lem}[Moment estimates]\label{estimate for mass for h}
Let $T>0$ be sufficiently small, independent of $\delta,L$, and $R$. Then there exists $K_{3}>0$, which is independent of $\delta\in\big(0,\frac{1}{2}\big)$, $L$, and $R$, such that for all $t \in [0,T]$ we have 
\begin{align}\label{estimate for moment for h}
M_{1,L}(x,t):=\int_{(0,\infty)}v
H_{L} (x,v,t)\der v &\leq \frac{K_{3}C_{0}}{1+|x|^{m-d}} \qquad \textup{ for }x\in\mathbb{R}\,, t \geq 0.
\end{align}
In general, if $p>\max\{2\gamma+1,2\}$, we have that
\begin{align}\label{estimate for all moment for h}
M_{n,L}(x,t):=\int_{(0,\infty)}v^{n}H_{L}(x,v,t)\der v &\leq \frac{K_{3}C_{0}}{1+|x|^{m-\frac{n+1}{\alpha}}} \qquad  \textup{ for }x\in\mathbb{R}\,, t \geq 0\,,
\end{align}
for $n\in[0,\max\{2\gamma,1\}]$. 
\end{lem}
\begin{proof}
We first recall that due to \eqref{Vestimate} we have $G_{L}(x,v,t) \leq \frac{C}{1+|X(x,v,t)|^m+v^p}$. 
We split the integral for $M_{n,L}$ as follows:
\[
 M_{n,L}(x,t)=\int_{0}^{v_{\max}(x,t)} v^n H_{L}(x,v,t)\,\der v + \int_{v_{\max}(x,t)}^{\infty} v^n H_{L}(x,v,t)\,\der v =: M_{n,1}(x,t) + M_{n,2}(x,t)\,,
\]
where $v_{\max}(x,t)=0$ if $x\leq 0$. From (\ref{Vestimate}) and since $\delta\in(0,\frac{1}{2})$, we have that $\frac{v}{2}\leq V(x,v,t)\leq \frac{3v}{2}$.
Due to \eqref{Xestimate}  and since $\delta\in(0,\frac{1}{2})$, it holds that $X(x,v,t)\geq x-(1+\delta)v^{\alpha}t\geq x-2v^{\alpha}t$. Then, by \eqref{vmaxestimate}, we further deduce that
\begin{align}\label{x estimate independent of delta}
 X(x,v_{\textup{max}}(x,t),t)\geq x-2v_{\textup{max}}(x,t)^{\alpha}t\geq x-2K_{max}xt^{\frac{m}{m-1}}\geq \frac{x}{2}\geq 0,
\end{align}
 for all $t\leq T$ if $T$ is sufficiently small. Thus, by Definition \ref{definition h for all l} and \eqref{form solution characteristics}, we obtain that 
\begin{align*}
 M_{n,1}(x,t) = &\int_0^{v_{\max}} v^n H_{L}(x,v,t)\der v = \int_0^{v_{\max}} v^n G_{L}(x,v_{\textup{max}}(x,t),t)\der v\\
 &\leq  
 \frac{ C v_{\max}^{n+1}}{1 + |x|^m + v_{\max}^p}  \leq \frac{C }{1+|x|^{m-\frac{n+1}{\alpha}}},
\end{align*}
for some constant $C>0$ which is independent of $\delta,L,$ and $R.$

To estimate $M_{n,2}$ we note that for all $x \in \R$ we have that
\begin{align*}
 M_{n,2}(x,t)\leq C  \int_0^{\infty} \frac{v^n}{1+v^{p}}\,\der v \leq C
\end{align*}
since $p-n>1$ by assumption. 

To obtain a decay for large $|x|$ we consider first  $x>1$ and use \eqref{xpos1}-\eqref{xpos2} to obtain  
\begin{align*}
 M_{n,2}(x,t)& \leq C \bigg( \int_0^{\big(\frac{x}{(1+2\delta)t}\big)^{\frac{1}{\alpha}}}\frac{v^n}{1+ (x-(1+\delta) v^{\alpha}t)^m+v^p}\der v  + \int_{\big(\frac{x}{(1+2\delta)t}\big)^{\frac{1}{\alpha}}}^{\infty} v^{n-p}\,\der v\bigg).
\end{align*}
Then, since $\delta<\frac{1}{2}$, we have that
\begin{align*}
 \int_{\big(\frac{x}{(1+2\delta)t}\big)^{\frac{1}{\alpha}}}^{\infty} v^{n-p}\,\der v
 \leq \frac{1}{p-n-1}\Big( \frac{(1+2\delta)t}{x}\Big)^{m-\frac{n+1}{\alpha}}\leq \frac{2^{m-\frac{n+1}{\alpha}}}{p-n-1}\Big( \frac{t}{x}\Big)^{m-\frac{n+1}{\alpha}}.
\end{align*}
For the other term, by using the change of variables $v=\big( \frac{x}{t}\big)^{\frac{1}{\alpha}} \xi$, we find
   \begin{align*}
  \int_0^{\big(\frac{x}{(1+2\delta)t}\big)^{\frac{1}{\alpha}}}\frac{v^n}{1+ (x- (1+\delta) v^{\alpha}t)^m+v^p}\der v
&\leq \big(\frac{x}{t}\big)^{\frac{1}{\alpha}}\int_{(0,\infty)}\frac{\big(\frac{x}{t}\big)^{\frac{n}{\alpha}}\xi^{n}}{1+x^{m}|1-(1+\delta)\xi^{\alpha}|^{m}+(\frac{x}{t})^{\frac{p}{\alpha}}\xi^{p}}\der  \xi\\
    &\leq \frac{1}{x^{m-\frac{1}{\alpha}}}\frac{1}{t^{\frac{1}{\alpha}}}\int_{(0,\infty)}\frac{\big(\frac{x}{t}\big)^{\frac{n}{\alpha}}\xi^{n}}{|1-(1+\delta)\xi^{\alpha}|^{m}+(\frac{1}{t})^{m}\xi^{p}}\der  \xi.
    \end{align*}
      Since $t\leq 1$ and  $m>1$ is even we have that the following holds
    \begin{align}\label{inequality for moment estimates independent of delta}
    3^{m+1}[|1-(1+\delta)\xi^{\alpha}|^{m}+\frac{1}{t^{m}}\xi^{p}]\geq 1+\frac{1}{t^{m}}\xi^{p},
    \end{align} 
    for $\xi\geq 0$. Indeed, if $|1-(1+\delta)\xi^{\alpha}|\geq \frac{1}{2}$, then \eqref{inequality for moment estimates independent of delta} follows. Otherwise, if  $|1-(1+\delta)\xi^{\alpha}|\leq \frac{1}{2}$, then $\frac{1}{2}\leq(1+\delta)\xi^{\alpha}\leq \frac{3}{2}$ and since $\delta<\frac{1}{2}$ it holds that $\xi^{\alpha}\geq \frac{1}{3}$. Thus, if we use in addition that $t\leq 1$, we have that
    \begin{align*}
        |1-(1+\delta)\xi^{\alpha}|^{m}+\frac{1}{t^{m}}\xi^{p}\geq \frac{1}{2t^{m}}\xi^{p}+\frac{1}{2t^{m}}\xi^{p}\geq \frac{1}{3^{m+1}}+\frac{1}{2t^{m}}\xi^{p}
    \end{align*}
    and thus (\ref{inequality for moment estimates independent of delta}) holds. It then follows that
        \begin{align*}
 \frac{1}{x^{m-\frac{1}{\alpha}}t^{\frac{1}{\alpha}}}\int_{(0,\infty)}\frac{\big(\frac{x}{t}\big)^{\frac{n}{\alpha}}\xi^{n}}{|1-\xi^{\alpha}|^{m}+(\frac{1}{t})^{m}\xi^{p}}\der  \xi&\leq \frac{C}{x^{m-\frac{1}{\alpha}}t^{\frac{1}{\alpha}}}\int_{(0,\infty)}\frac{\big(\frac{x}{t}\big)^{\frac{n}{\alpha}}\xi^{n}}{1+(\frac{1}{t})^{m}\xi^{p}}\der  \xi\\
   & \leq \frac{C}{x^{m-\frac{n+1}{\alpha}}}\int_{(0,\infty)}\frac{\eta^{n}}{1+\eta^{p}}\der  \eta\leq \frac{C}{x^{m-\frac{n+1}{\alpha}}},
    \end{align*}
for some constant $C>0$ which is independent of $\delta\in\big(0,\frac{1}{2}),L,$ and $R$.

In the case $x\leq 0$ we can  use \eqref{xnegative} to obtain the estimate similarly as above without the  need to split the integral.
Estimate 
    (\ref{estimate for moment for h}) follows from (\ref{estimate for all moment for h}) by choosing $n=1$ and the fact that $d$ in (\ref{we need function to be even}) satisfies $d>\frac{2}{\alpha}$.
\end{proof}
\begin{rmk}
 Let $T>0$ be sufficiently small. With similar computations as the ones used in Lemma \ref{estimate for mass for h}, we can prove that 
 \begin{align}\label{bound needed for the cauchy sequence}
    \int_{(0,\infty)}\frac{v^{n}}{1+|x|^{m}+v^{p}}\der v &\leq \frac{K_{3}}{1+|x|^{m-\frac{n+1}{\alpha}}},
 \end{align}
 for all $x\in\mathbb{R},$ $t\in[0,T]$, and for $n\in[0,\max\{2\gamma,1\}]$.
\end{rmk}

  We now define the function for which we will prove is a supersolution for the problem (\ref{inductive sequence}).
\begin{defi}\label{definition supersolution}
Let $\delta\in\big(0, \frac{1}{4}\big]$ and   $L=4K_{1}K_{2}K_{3}C_{0}$, where $C_{0}$ is as in \eqref{form solution characteristics}, $K_{1}$ is as in \eqref{condition on gamma existence}, $K_{2}$ is as in Lemma \ref{L.Hproperties}, and $K_{3}$ is as in Lemma \ref{estimate for mass for h}. Denote by $H(x,v,t):=H_{L}(x,v,t)$, where $H_{L}$ is as in Definition \ref{definition h for all l}. Moreover, let $B_t$ to be
\begin{align}\label{Btdef}
     B_t:=e^{\lambda t + \lambda t^{\frac{\alpha+1-\gamma}{\alpha}}} 
    \end{align}
for some $\lambda>0$. 
    Then we define the function $G$ via
    \begin{align}\label{definition of supersol supersol}
    G(x,v,t)=B_{t}H(x,v,t).
\end{align}
\end{defi}

As mentioned before, with this construction $G$ is decreasing in $v$ for fixed $x$ outside possibly a critical region where 
$x \in [(1-2\delta) v^{\alpha}t, (1+2\delta) v^{\alpha}t]$. In the following we will have to deal with this region separately. 

Our key result is the following.

\begin{prop}\label{proposition about supersolution}
Let $T>0$ be sufficiently small and $\delta\in\big(0,\frac{1}{4}\big]$. There exists a sufficiently large $\lambda>0$, which depends only on $C_{0}$ and the parameters $m,\gamma,\alpha$ such that if 
  $f_n \leq G$, 
 where the sequence $\{f_{n}\}_{n\in\mathbb{N}}$ was defined in (\ref{inductive sequence}), then  $f_{n+1}\leq G$, for all $n \in \N$ and all $t\in[0,T]$.
\end{prop}
\begin{rmk}\label{choice of time}
Since the constants will play an important role in our proof,   it is worthwhile to mention for clarity which are the constants that the parameters in Proposition \ref{proposition about supersolution} depend on. Let $C_{0}$ be as in \eqref{form solution characteristics}, $K_{1}$ as in \eqref{condition on gamma existence}, $K_{2}$ as in Lemma \ref{L.Hproperties}, and $K_{3}$ be as in Lemma \ref{estimate for mass for h}. Notice that these constants do not depend on $\delta$, $L$ or $R$ from Proposition \ref{properties of the characteristics}. In order to prove that $G$ is a supersolution, we take   $L=4K_{1}K_{2}K_{3}C_{0}$. 
     We then take $\lambda$ to be sufficiently large depending on  $C_{0}, K_{1}, K_{2}$, and $K_{3}$. We then take $T$ to be such that $\max\{T,T^{\frac{\alpha+1-\gamma}{\alpha}}\}\leq \frac{\ln(2)}{2\lambda},$ which implies that $B_{t}\leq 2$, for all $t\in[0,T]$, where $B_{t}$ was defined in \eqref{Btdef}.

\end{rmk}

Before we begin with the proof of Proposition \ref{proposition about supersolution}, it is worthwhile to notice that we have some moment bounds for the function $G$ in Definition \ref{definition supersolution} as a direct consequence of Lemma \ref{estimate for mass for h}. 
\begin{lem}\label{estimate for mass}
Let $T>0$ be sufficiently small with $K_{3}$ as in Lemma \ref{estimate for mass for h}. Then for all $t \in [0,T]$ we have 
\begin{align}\label{estimate for moment}
M_{1}(x,t):=\int_{(0,\infty)}v
G (x,v,t)\der v &\leq \frac{K_{3}C_{0}B_t}{1+|x|^{m-d}} \qquad \textup{ for }x\in\mathbb{R},
\end{align}
where $B_t$ was defined in \eqref{Btdef}.
In general, if $p>\max\{2\gamma+1,2\}$, we have that
\begin{align}\label{estimate for all moment}
M_{n}(x,t):=\int_{(0,\infty)}v^{n}G(x,v,t)\der v &\leq \frac{K_{3}C_{0}B_t}{1+|x|^{m-\frac{n+1}{\alpha}}} \qquad  \textup{ for }x\in\mathbb{R}\,, t \geq 0\,,
\end{align}
for $n\in[0,\max\{2\gamma,1\}]$. 
\end{lem}

We now focus on proving Proposition \ref{proposition about supersolution}.

\bigskip
\begin{proof}[Proof of Proposition \ref{proposition about supersolution}]

We now prove that $G$ is a supersolution of the problem (\ref{inductive sequence}), that is we show
\begin{align}\label{main ingredient supersolution}
& \partial_{t}G(x,v,t)+v^{\alpha}\partial_{x}G(x,v,t)-\int_{0}^{\frac{v}{2}}K(v-v',v')G(x,v-v',t)f_{n}(x,v',t)\der v'\nonumber\\
    &  +\int_{0}^{\infty}K(v,v')G(x,v,t)f_{n}(x,v',t)\der v'\geq 0.
    \end{align}
We prove (\ref{main ingredient supersolution}) by showing first that
\begin{align}\label{second ingredient supersolution}
\partial_{t}G(x,v,t)+v^{\alpha}\partial_{x}G(x,v,t)+\frac{{L}v^{\gamma}}{1+|x|^{m-d}}\xi_{R}(v)\partial_{v}G(x,v,t)\geq 0, 
    \end{align}
for any ${L}>0$ and where $d$ was defined in (\ref{we need function to be even}), and then that
    \begin{align}\label{third ingredient supersolution}
& \partial_{t}G(x,v,t)+v^{\alpha}\partial_{x}G(x,v,t)-\int_{0}^{\frac{v}{2}}K(v-v',v')G(x,v-v',t)f_{n}(x,v',t)\der v'\nonumber\\
    &  +\int_{0}^{\infty}K(v,v')G(x,v,t)f_{n}(x,v',t)\der v'\nonumber\\
   & \geq \partial_{t}G(x,v,t)+v^{\alpha}\partial_{x}G(x,v,t)+\frac{{L}v^{\gamma}}{1+|x|^{m-d}}\xi_{R}(v)\partial_{v}G(x,v,t),
    \end{align}
for $L>0$ as in Definition \ref{definition supersolution}.

\bigskip

We now prove (\ref{second ingredient supersolution}).

Assume $x\not\in[(1-2\delta) v^{\alpha}t,(1+2\delta) v^{\alpha}t]$ and $\partial_{v}G_{L}\leq 0$ or $x\in[(1-2\delta) v^{\alpha}t,(1+2\delta) v^{\alpha}t]$. 
Then $G=B_tG_{L}$ and thus, using (\ref{eq char}), it holds with $c_{\alpha}:=\frac{\alpha+1-\gamma}{\alpha}$ that
\begin{align*}
   & \partial_{t}G(x,v,t)+v^{\alpha}\partial_{x}G(x,v,t)+\frac{{L}v^{\gamma}}{1+|x|^{m-d}}\xi_{R}(v)\partial_{v}G(x,v,t)\\
    &=   B_t\bigg(\partial_{t}G_{L}(x,v,t)+v^{\alpha}\partial_{x}G_{L}(x,v,t)+\frac{{L}v^{\gamma}}{1+|x|^{m-d}}\xi_{R}(v)\partial_{v}G_{L}(x,v,t)\bigg)\\
    &\qquad +\lambda(1+c_{\alpha} t^{\frac{\alpha+1-\gamma}{\alpha}-1})B_t G_{L}\geq \lambda\textup{e}^{\lambda t}G_{L}(x,v,t)\geq 0.
\end{align*}
Notice that we did not use the contribution of the term $t^{\frac{\alpha+1-\gamma}{\alpha}}$ in the computations. This term will be needed  later in the proof. 

Assume now $x\not\in[(1-2\delta) v^{\alpha}t, (1+2\delta) v^{\alpha}t]$ and $\partial_{v}G_{L} >0$ such that we have $G(x,v,t)=B_{t}G_{L}(x,v_{\max}(x,t),t)$. Then
\begin{align}\label{general function supersolution}
   & \partial_{t}G(x,v,t)+v^{\alpha}\partial_{x}G(x,v,t)+\frac{{L}v^{\gamma}}{1+|x|^{m-d}}\xi_{R}(v)\partial_{v}G(x,v,t)\nonumber\\
    &= B_t     \bigg(\partial_{t}G_{L}(x,v_{\textup{max}}(x,t),t)+v^{\alpha}\partial_{x}G_{L}(x,v_{\textup{max}}(x,t),t)\bigg)\\
    &\qquad +\lambda(1+c_{\alpha} t^{\frac{\alpha+1-\gamma}{\alpha}-1})B_t G_{L}(x,v_{\textup{max}}(x,t),t)\nonumber\\
    &\geq B_t \bigg(\partial_{t}G_{L}(x,v_{\textup{max}}(x,t),t)+v^{\alpha}\partial_{x}G_{L}(x,v_{\textup{max}}(x,t),t)\bigg).
\end{align}
By the choice of $v_{\textup{max}}(x,t)$ we have that $v\leq v_{\textup{max}}(x,t)$ in the region where $\partial_{v}G_{L}(x,v,t)>0$. Moreover, from Proposition \ref{derivatives in x variable}, we know that $\partial_{x}G_{L}(x,v_{\textup{max}}(x,t),t)\leq 0$. Thus, from (\ref{general function supersolution}) and (\ref{eq char}), we further obtain that
\begin{align*}
   & \partial_{t}G(x,v,t)+v^{\alpha}\partial_{x}G(x,v,t)+\frac{{L}v^{\gamma}}{1+|x|^{m-d}}\xi_{R}(v)\partial_{v}G(x,v,t)\\
    &\geq B_t  \bigg(\partial_{t}G_{L}(x,v_{\textup{max}}(x,t),t)+v_{\textup{max}}^{\alpha}\partial_{x}G_{L}(x,v_{\textup{max}}(x,t),t)\bigg)= 0
\end{align*}
and \eqref{second ingredient supersolution} follows.

We are thus left to prove that (\ref{third ingredient supersolution}) holds.
We analyze the cases when $x\in[(1-2\delta) v^{\alpha}t, (1+2\delta) v^{\alpha}t]$ and $x\not\in[(1-2\delta) v^{\alpha}t,(1+2\delta) v^{\alpha}t]$ separately.

\textbf{Proof of  (\ref{third ingredient supersolution}) for} $x\not\in[(1-2\delta) v^{\alpha}t, (1+2\delta) v^{\alpha}t]$.

We have that
\begin{align}\label{positivity of the additional term part two}
-\int_{0}^{\frac{v}{2}}&K(v-v',v')G(x,v-v',t)f_{n}(x,v',t)\der v' +\int_{0}^{\infty}K(v,v')G(x,v,t)f_{n}(x,v',t)\der v'\nonumber\\
    &\geq  -\int_{0}^{\frac{v}{2}}K(v,v')[G(x,v-v',t)-G(x,v,t)]f_{n}(x,v',t)\der v'\nonumber\\
    &\qquad \qquad +\int_{0}^{\frac{v}{2}}[K(v,v')-K(v-v',v')]G(x,v-v',t)f_{n}(x,v',t)\der v'.
    \end{align}
    Since $G,f_{n}\geq 0$ and $K(v-v',v')\leq K(v,v')$ when $v'\in[0,\frac{v}{2})$ by (\ref{growth condition on the kernel}), it holds that 
    \begin{align}\label{positivity of the additional term}
        \int_{0}^{\frac{v}{2}}[K(v,v')-K(v-v',v')]G(x,v-v',t)f_{n}(x,v',t)\der v'\geq 0.
    \end{align}
We know that $K(v,v')\leq K_{1}(v^{\gamma}+v'^{\gamma})$, where $K_{1}$ is as in (\ref{condition on gamma existence}). Thus, when $v'\leq\frac{v}{2}$, it holds that $K(v,v')\leq 2K_{1}v^{\gamma}$. Without loss of generality  we assume in the following that $K(v,v')\leq v^{\gamma}$ when $v'\leq\frac{v}{2}$. This is in order to simplify the notation but we allow $L$ in Definition \ref{definition supersolution} to depend on $K_{1}$. Additionally, by (\ref{definition of supersol supersol}), it holds that $\partial_{v}G\leq 0$ and thus $G(x,v-v',t)-G(x,v,t)\geq 0$, for $v\in(0,\frac{v}{2})$. Since $f_n \leq G$  and  with (\ref{positivity of the additional term}) we deduce that
    \begin{align*}
 -&\int_{0}^{\frac{v}{2}}v^{\gamma}[G(x,v-v',t)-G(x,v,t)]f_{n}(x,v',t)\der v'\\
&    \geq  -v^{\gamma}\int_{0}^{\frac{v}{2}}\big(G(x,v-v',t)-G(x,v,t)\big)G(x,v',t)\der v'.
    \end{align*}
  We use the following notation
    \begin{align*}
    I_{1}:=v^{\gamma}\int_{0}^{\frac{v}{2}}\big(G(x,v-v',t)-G(x,v,t)\big)G(x,v',t)\der v' =\xi_{R}(v)I_{1}+(1-\xi_{R}(v))I_{1}
    \end{align*}
with $\xi_{R}$  as in \eqref{chirdef}. 
Assume that $H$, as in Definition \ref{definition supersolution},  satisfies
    \begin{align}\label{supersolution}
         \xi_{R}(v)v^{\gamma}\int_{0}^{\frac{v}{2}}\big(H(x,v-v',t)- H(x,v,t)\big)G(x,v',t)\der v'\leq & -\frac{{L} v^{\gamma}}{1+|x|^{m-d}} \xi_{R}(v)\partial_{v}   H(x,v,t)\nonumber\\
         &+CB_{t}    H(x,v,t)
    \end{align}
    and 
    \begin{align}
 (1-\xi_{R}(v))v^{\gamma}\int_{0}^{\frac{v}{2}}\big(H(x,v-v',t)-H(x,v,t)\big)G(x,v',t)\der v'\leq CB_{t} H(x,v,t),\label{alter lambda part two}
\end{align}
for  $L=4K_{1}K_{2}K_{3}C_{0}$ as in Definition \ref{definition supersolution}. 
Then
\begin{align*}
 \partial_t& G + v^{\alpha} \partial_{x}G - v^{\gamma}\int_{0}^{\frac{v}{2}}\big(G(x,v-v',t)-G(x,v,t)\big)
 G(x,v',t)\der v'\\
 &\geq B_t  \Big ( \partial_t  H + v^{\alpha} \partial_{x} H -
 v^{\gamma}\int_{0}^{\frac{v}{2}}\big( H (x,v-v',t)-
H (x,v,t)\big) G(x,v',t)\der v'\Big) + \lambda  e^{\lambda t+\lambda t^{\frac{\alpha+1-\gamma}{\alpha}}}  H \\
 &\geq  B_t   \Big ( \partial_t  H + v^{\alpha} \partial_{x} H +\frac{{L} v^{\gamma}}{1+|x|^{m-d}} \xi_{R}(v)\partial_{v}   H(x,v,t)\Big) \\
 &+ B_t    \big( -2CB_{t}+ \lambda \big)  H  \geq \partial_t G+ v^{\alpha} \partial_x G + \frac{{L} v^{\gamma}}{1+|x|^{m-d}} \xi_R(v) \partial_v G
\end{align*}
if $\lambda \geq 2CB_{t}$, which holds true if $\lambda=4C$ and if $\max\{t,t^{\frac{\alpha+1-\gamma}{\alpha}}\}\leq \frac{\ln(2)}{2\lambda}$. This proves \eqref{third ingredient supersolution}.

\bigskip
{\bf Proof of (\ref{alter lambda part two}):}.   We have the following cases

\bigskip
{ \bf Case $1.$ $ a)$  $x> 0$  and $v\leq v_{\textup{max}}(x,t)$:} In this case $ H(x,v-v',t)= H(x,v,t)$, for $v'\in(0,\frac{v}{2}]$. Thus, (\ref{alter lambda part two}) holds.

\bigskip
{\bf Case $1.$ $ b)$  $x> 0$ and $v_{\textup{max}}(x,t)\leq v\leq 2v_{\textup{max}}(x,t)$:} The proof of this case is the same as for Case $1.$ $b)$ when proving that (\ref{supersolution}) holds. We thus postpone its proof.

\bigskip

{\bf Case $2.$ Either $\{x> 0$ and $ 2v_{\textup{max}}(x,t)\leq v\}$ or $\{x\leq 0\}$:} 
Since $v\leq 2R$ and $\gamma\geq 0$, it follows, using \eqref{H1} and Lemma \ref{estimate for mass}, that
\begin{align}
    (1-\xi_{R}(v))I_{1}\textup{e}^{-\lambda t}&\leq (2R)^{\gamma}\int_{0}^{\frac{v}{2}}\big(H(x,v-v',t)- H(x,v,t)\big)G(x,v',t) \der v'\nonumber\\
    &\leq (2R)^{\gamma}\int_{0}^{\frac{v}{2}} H(x,v-v',t)G(x,v',t) \der v’ \nonumber \\
    &\leq K_{2}(2R)^{\gamma} H(x,v,t)\int_{0}^{\frac{v}{2}}G(x,v',t)\der v'\nonumber\\
 &\leq\frac{ C_{0}K_{2}K_{3}(2R)^{\gamma} B_{t} H(x,v,t)}{1+|x|^{m-\frac{1}{\alpha}}}\leq C B_{t}H(x,v,t).\label{alter lambda}
\end{align}



\begin{rmk}\label{remark dependence constants 1}
    Notice that the constant $C$ in \eqref{alter lambda} depends on $R$ from Proposition \ref{properties of the characteristics}, on the constant $K_{2}$ from \eqref{H1}, and on the constant $K_{3}$ from  Lemma \ref{estimate for mass}. However, since $L$ is fixed in Definition \ref{definition supersolution}, $R$ depends only on $K_{1}, K_{2}, K_{3}$ and $C_{0}$.
\end{rmk}

\bigskip
{\bf Proof  of (\ref{supersolution}):}  We have the following cases



\bigskip

{\bf Case $1.$ $a)$: $x>0$ and $v\leq v_{\textup{max}}(x,t)$:}  Notice that in this case $ H(x,v-v',t)= H(x,v,t)$, for $v'\in(0,\frac{v}{2}]$, and $\partial_{v} H(x,v,t)=0$. Thus, (\ref{supersolution}) holds.

\bigskip

{\bf Case $1.$ $b)$: $x>0$ and   $v_{\textup{max}}(x,t)\leq v\leq 2v_{\textup{max}}(x,t)$:} 
We divide the integral on the left hand side of \eqref{supersolution}  into the region $v' \in (0,v-v_{\text{max}})$ and $v' \in (v-v_{\text{max}},\frac{v}{2})$, respectively.  For $v' \in (v-v_{\text{max}},\frac{v}{2})$, using
 $\frac{v}{2}\leq v_{\textup{max}}(x,t)$ and Lemma \ref{estimate for mass} with $n=0$, we obtain that
\begin{align}
\int_{v-v_{\textup{max}}(x,t)}^{\frac{v}{2}}\big(  H(x,v-v',t)- H(x,v,t)\big) G(x,v',t)\der v'
&\leq  H(x,v_{
\textup{max}}(x,t),t)\int_{0}^{v_{\textup{max}}}G(x,v',t)\der v'\nonumber\\
& \leq \frac{ C_{0}K_{3} B_{t}H(x,v_{
\textup{max}}(x,t),t)}{1+|x|^{m-\frac{1}{\alpha}}}.\label{firstineq}
\end{align}
We have $0\leq v^{\alpha}\leq 2 v_{\textup{max}}^{\alpha}(x,t)\leq 2K_{max} xt^{\frac{1}{m-1}}$, with $K_{max}$ as in \eqref{vmaxestimate}.

Moreover, using \eqref{H2} it follows that
    \begin{align*}
        v^{\gamma}\int_{v-v_{\textup{max}}(x,t)}^{\frac{v}{2}}\big( H(x,v-v',t)- H(x,v,t)
        \big) G(x,v',t)\der v'\leq   \frac{ 2^{\gamma}C_{0}K_{2}K_{3}B_{t}H(x,v,t)v_{\textup{max}}(x,t)^{\gamma}}{1+x^{m-\frac{1}{\alpha}}}.
    \end{align*}
Since $v_{\textup{max}}(x,t)^{\alpha}\leq 2K_{max} xt^{\frac{1}{m-1}}$, with $K_{max}$ as in \eqref{vmaxestimate}, we further obtain that
    \begin{align*}
     \frac{H(x,v,t)v_{\textup{max}}(x,t)^{\gamma}}{1+x^{m-\frac{1}{\alpha}}}\leq  \frac{CH(x,v,t)x^{\frac{\gamma}{\alpha}}t^{\frac{\gamma}{\alpha(m-1)}}}{1+x^{m-\frac{1}{\alpha}}}\leq Ct^{\frac{\gamma}{\alpha(m-1)}} H(x,v,t),
     \end{align*}
     since $m>\frac{\gamma+1}{\alpha}$. Thus
         \begin{align}\label{dependence constants part one}
        v^{\gamma}\int_{v-v_{\textup{max}}(x,t)}^{\frac{v}{2}}\big(H(x,v-v',t)- H(x,v,t)\big)
        G(x,v',t)\der v'\leq CB_{t}t^{\frac{\gamma}{\alpha(m-1)}} H(x,v,t).
    \end{align}
We now estimate the integral
\begin{align}
J:=\int_{0}^{v-v_{\textup{max}}(x,t)}\big( H(x,v-v',t)- H(x,v,t)\big)G(x,v',t)\der v'.\label{second case integral}
\end{align}
As before, using \eqref{H1} and  Lemma \ref{estimate for mass}, we find
\begin{align*}
J\leq \int_{0}^{v{-}v_{\textup{max}}(x,t)}H(x,v{-}v',t) G(x,v',t)
\der v'\leq & \int_{0}^{v-v_{\textup{max}}(x,t)}K_{2}  H(x,v,t)G(x,v',t)\der v'\\
&\leq  
\frac{ C_{0}K_{2}K_{3}B_{t}H(x,v,t)}{1+x^{m-\frac{1}{\alpha}}}\nonumber
\end{align*}
and this implies, since $v\leq 2 v_{\textup{max}}(x,t)$, that
\begin{align}
v^{\gamma}\int_{0}^{v{-}v_{\textup{max}}(x,t)}\big(H(x,v{-}v',t)- H(x,v,t)\big)G(x,v',t)\der v&\leq C B_{t}H(x,v,t)\frac{v_{\textup{max}}(x,t)^{\gamma}}{1+x^{m-\frac{1}{\alpha}}}\nonumber\\
&\leq CB_{t} H(x,v,t)\frac{x^{\frac{\gamma}{\alpha}}t^{\frac{\gamma}{\alpha(m-1)}}}{1+x^{m-\frac{1}{\alpha}}}.\label{dependence constants part two}
\end{align}
Since $m>\frac{\gamma+1}{\alpha}$ it follows that (\ref{supersolution})  holds in this case.

\begin{rmk}\label{remark dependence constants 2}
       Notice that the constant $C$ in \eqref{dependence constants part one} and \eqref{dependence constants part two} depends only on the constant $K_{2}$ from \eqref{H1} and on the constants $C_{0},K_{3}$ from  Lemma \ref{estimate for mass}. 
\end{rmk}

\bigskip
{\bf Case $1.$ $c)$: $x>0$ and $2v_{\max}(x,t) \leq v$ or $x<0$:}  Notice that the computations used in Case $1.$ $ b)$ hold for any $v\leq C_{l} v_{\textup{max}}(x,t)$, for some fixed constant $C_{l}>0$. We can thus assume without loss of generality in this case that $v> C_{l} v_{\textup{max}}(x,t)$  and that $v\geq R$ because of the presence of $\xi_{R}(v)$ in (\ref{supersolution}).

We have that $v-v'\in (\frac{v}{2},v)$, for $v'\in(0,\frac{v}{2})$. It holds
   \begin{align*}
        \int_{0}^{\frac{v}{2}}\big( H(x,v-v',t)- H(x,v,t)\big) G(x,v',t)\der v'
        =    - \int_{0}^{\frac{v}{2}}G(x,v',t)\int_{v-v'}^{v}\partial_{ v} H(x,\tilde v,t)\der \tilde v\der v'.
    \end{align*}
We can  then use (\ref{H4}), \eqref{H3} and Lemma \ref{estimate for mass} with $n=1$ in order to deduce that
 \begin{align}
    - v^{\gamma}\xi_{R}(v)&\int_{0}^{\frac{v}{2}}G(x,v',t)\int_{v-v'}^{v}\partial_{v} H(x,\tilde v,t)\der \tilde v\der v'\nonumber\\
    &\leq  -K_{2}v^{\gamma}\xi_{R}(v)\partial_{v} H(x,v,t)\int_{0}^{\frac{v}{2}}v'G(x,v',t)\leq -\frac{C_{0}K_{2}K_{3}B_{t}v^{\gamma}\xi_{R}(v)\partial_{v} H(x,v,t)}{1+|x|^{m-d}}\nonumber\\
    &\leq -\frac{L v^{\gamma}\xi_{R}(v)\partial_{v} H(x,v,t)}{1+|x|^{m-d}},\label{the choice of L}
    \end{align}
    where in the last inequality we used the definition of $L$ in Definition \ref{definition supersolution} and that $B_{t}\leq 2$ by Remark \ref{choice of time}. Thus (\ref{supersolution}) holds in this case.

\bigskip

\textbf{Proof of  (\ref{third ingredient supersolution}) for} $x\in[(1-2\delta) v^{\alpha}t, (1+2\delta) v^{\alpha}t]$ or, alternatively, $v^{\alpha}t\in[\frac{x}{1+2\delta}, \frac{x}{1-2\delta}].$ 

 We will  assume  that $v\geq 1$ and $\gamma>1$, since the other cases are similar but easier to treat.

In order to prove (\ref{third ingredient supersolution}) we will first  show that
\begin{align}\label{step one problematic region supersolution}
    &\int_{0}^{\frac{v}{2}}K(v,v')f_{n}(x,v',t)|H(x,v-v',t)-H(x,v,t)|\der v'\leq CL(1+t^{\frac{\alpha+1-\gamma}{\alpha}-1}) H(x,v,t)
    \end{align}
and then that
\begin{align}\label{step two problematic region supersolution}
      &\frac{v^{\gamma}}{1+|x|^{m-d}}|\partial_{v}H(x,v,t)|\leq CLt^{\frac{\alpha+1-\gamma}{\alpha}-1} H(x,v,t),
\end{align}
for $H$ and $L$ as in Definition \ref{definition supersolution}. If  (\ref{step one problematic region supersolution}) and (\ref{step two problematic region supersolution}) hold,
then   (\ref{third ingredient supersolution}) holds. Indeed, arguing as in \eqref{positivity of the additional term part two}, \eqref{positivity of the additional term} and using $f_n \leq G$, we find 
\begin{align*}
(*):=\partial_t G& + v^{\alpha}\partial_{x} G -\int_0^{\frac{v}{2}} K(v-v',v') G(x,v-v',t) f_n(x,v',t)\,dv' \\
&\qquad + \int_0^{\infty} K(v,v') G(x,v,t)f_n(x,v',t)\,dv'\\
&\geq B_t \bigg ( \partial_t H + v^{\alpha}  \partial_{x} H +\frac{{L}v^{\gamma}}{1+|x|^{m-d}}\xi_{R}(v)\partial_{v}H(x,v,t)\bigg)\\
 & - B_t \frac{{L}v^{\gamma}}{1+|x|^{m-d}}|\partial_{v}H(x,v,t)|- B_t v^{\gamma}\int_{0}^{\frac{v}{2}}| H (x,v-v',t)-
  H (x,v,t)|f_n(x,v',t)\der v' \\
 &\qquad + \lambda (c_{\alpha}t^{\frac{\alpha+1-\gamma}{\alpha}-1}+1) B_t  H .
 \end{align*}
Using now (\ref{step one problematic region supersolution}) and (\ref{step two problematic region supersolution}), we can conclude that
\begin{align*}
 (*)&\geq B_t  \bigg ( \partial_t H + v^{\alpha}  \partial_{x} H +\frac{{L}v^{\gamma}}{1+|x|^{m-d}}\xi_{R}(v)\partial_{v}H(x,v,t)\bigg)\\
 & \qquad + B_t \big[ - CL-2LCt^{\frac{\alpha+1-\gamma}{\alpha}-1} + \lambda (c_{\alpha}t^{\frac{\alpha+1-\gamma}{\alpha}-1}+1)\big]  H \\
 & \geq \partial_t G + v^{\alpha} \partial_x G +\frac{{L}v^{\gamma}}{1+|x|^{m-d}}\xi_{R}(v)\partial_{v}G(x,v,t)
\end{align*}
if $\lambda$ is sufficiently large. 

\begin{rmk}
    It is worthwhile to mention that $L$ depends only on $K_{1}$ from \eqref{condition on gamma existence}, $K_{2}$ from Lemma \ref{L.Hproperties}, and $K_{3}$ from  Lemma \ref{estimate for mass}, see Definition \ref{definition supersolution}.
\end{rmk}

We now prove that (\ref{step one problematic region supersolution}) holds. Let $\eta\in(0,1)$ be fixed and sufficiently small. We want to bound the following terms.
    \begin{align*}
    &\int_{\eta v}^{\frac{v}{2}}K(v,v')f_{n}(x,v',t)|H(x,v-v',t)-H(x,v,t)|\der v'\\
    &+\int_{0}^{\eta v}K(v,v')f_{n}(x,v',t)|H(x,v-v',t)-H(x,v,t)|\der v'=:J_{1}+J_{2}.
\end{align*}
We analyze each term separately. We notice that, since $\delta<1$ and $t$ is sufficiently small, we have that $v^{\alpha}\geq \frac{x}{(1+2\delta)t}\geq 2K_{max}xt^{\frac{1}{m-1}}\geq 2v_{\textup{max}}^{\alpha}$ in this region, where $K_{max}$ is as in \eqref{vmaxestimate}. Thus we can assume in all the following that 
\begin{align}\label{v and v' are correct region}
v\geq 2 v_{\textup{max}} \textup{ and } v'\geq v_{\textup{max}}, \textup{ for all } v'\in\Big[\frac{v}{2},v\Big].
\end{align}
Moreover, by (\ref{form solution characteristics}), it holds that $G(x,v,t)= \frac{C_{0}B_{t}}{1+|X|^{m}+V^{p}}\leq \frac{C_{0}B_{t}}{1+V^{p}}\leq \frac{2^{p}C_{0}B_{t}}{1+v^{p}}$. Using \eqref{H estimate problematic region} and the fact that $B_{t}\leq 2$, for $t\leq T$, we have 
\begin{align*}
   J_{1}&\leq C \int_{\eta v}^{\frac{v}{2}}v^{\gamma}[H(x,v-v',t)+H(x,v,t)]G(x,v',t)\der v'\leq C H(x,v,t) \int_{\eta v}^{\frac{v}{2}}\frac{v^{\gamma}}{1+v'^{p}}\der v'\\
    &\leq C(\eta)H(x,v,t)\int_{0}^{\infty}\frac{v'^{\gamma}}{1+v'^{p}}\der v'\leq C H(x,v,t).
\end{align*}
For the second term, it holds that
\begin{align*}
   J_{2}\leq Cv^{\gamma}\int_{0}^{\eta v}|H(x,v-v',t)-H(x,v,t)|G(x,v',t)\der v'\,.
\end{align*}
We remember we are in the region where $v^{\alpha}t\in[\frac{x}{1+2\delta}, \frac{x}{1-2\delta}]$ and thus $v'^{\alpha}\leq \eta^{\alpha}v^{\alpha}\leq \frac{\eta^{\alpha}x}{(1-2\delta)t}$. Due to \eqref{Xestimate}, it follows that
\begin{align*}
  X(x,v',t)\geq x-(1+\delta)v'^{\alpha}t  \geq x-\frac{(1+\delta)\eta^{\alpha}x}{1-2\delta}=  \frac{1-2\delta-(1+\delta)\eta^{\alpha}}{1-2\delta}x,
\end{align*}
for all $v'\in[0,\eta v]$. Since $\delta\leq\frac{1}{4}$ it follows that we can choose $\eta$ to be sufficiently small, but independent of $\delta$, such that $1-2\delta-(1+\delta)\eta^{\alpha}\geq \frac{1}{4}$ and thus
\begin{align*}
   X(x,v',t) \geq (1-2\delta-(1+\delta)\eta^{\alpha})x\geq \frac{x}{4},
\end{align*}
for all $v'\in[0,\eta v]$. Thus, since $v^{\alpha}t\in[\frac{x}{1+2\delta}, \frac{x}{1-2\delta}]$ and $\delta\leq\frac{1}{4}$, we have that  $X(x,v',t)\geq \frac{v^{\alpha}t}{8}$. It follows that
\begin{align*}
   J_{2}&\leq CB_{t}v^{\gamma}\int_{0}^{\eta v}\frac{\int_{v-v'}^{v}|\partial_{\tilde{v}}H(x,\tilde{v},t)|\der \tilde{v}\der v'}{1+(v^{\alpha}t)^{m}+v'^{p}}\leq C v^{\gamma}\int_{0}^{\eta v}\frac{\int_{v-v'}^{v}|\partial_{\tilde{v}}H(x,\tilde{v},t)|\der \tilde{v}\der v'}{1+(v^{\alpha}t)^{m}+v'^{p}}
\end{align*}
since $B_{t}\leq 2$. Let now $\tilde{v}\in[v-v',v]$, with $v'\in[0,\eta v].$ We have that 
\begin{align*}
|\partial_{\tilde{v}}H(x,\tilde{v},t)|\leq C_{0}\frac{|m|X|^{m-2}X\partial_{\tilde{v}}X+V^{p-1}\partial_{\tilde{v}}V|}{(1+|X|^{m}+V^{p})^{2}}(x,\tilde{v},t)\leq C_{0}\frac{|m|X|^{m-2}X\partial_{\tilde{v}}X+V^{p-1}\partial_{\tilde{v}}V|}{(1+V^{p})^{2}}(x,\tilde{v},t).
\end{align*}

Assume $v^{\gamma-1}t\geq 1$ and remember we are in the case when $\gamma>1$ and $v\geq 1$. Since $v'\leq \eta v$ it holds that $\frac{v}{2} \leq \tilde{v}\leq v$ for $\tilde{v}\in[v-v',v]$. In addition, by \eqref{v and v' are correct region}, we have that $\tilde{v}\geq v_{\textup{max}}$. If $\tilde{v}\geq \frac{x}{1+2\delta}$, then from (\ref{dvVestimate problematic region}), we have that
\begin{align}\label{estimate both cases}
    |m|X|^{m-2}X\partial_{\tilde{v}}X+V^{p-1}\partial_{\tilde{v}}V|\leq C L \Big( |X(x,\tilde v, t)|^{m-1}\tilde{v}^{\alpha-1}t+\tilde{v}^{p+\gamma-2}t\Big).
    \end{align}
    Otherwise, if  $\tilde{v}< \frac{x}{1+2\delta}$, then \eqref{estimate both cases} still holds since $v^{\gamma-1}t\geq 1$.

 Moreover, since $\frac{v}{2} \leq \tilde{v}\leq v$ for $\tilde{v}\in[v-v',v]$, it holds that $|X(x,\tilde v,t)|\leq Cv^{\alpha}t\leq C v^{\alpha}$ and thus $|X(x,\tilde v, t)|^{m-1}\tilde{v}^{\alpha-1}t\leq C v^{p-1}t\leq C v^{p+\gamma-2}t$ since $\gamma>1$. Thus
\begin{align}\label{estimate for the derivative of g supersolution}
|\partial_{\tilde{v}}H(x,\tilde{v},t)|\leq \frac{CLv^{p+\gamma-2}t}{(1+v^{p})^{2}}
\end{align}
and thus it holds that
\begin{align*}
   J_{2}\leq CL v^{2(\gamma-1)}\frac{v^{p}t}{(1+v^{p})^{2}}\int_{0}^{\eta v}\frac{v'\der v'}{1+(v^{\alpha}t)^{m}+v'^{p}}.
\end{align*}

By making the change of variables $v'=(1+v^{p}t^{m})^{\frac{1}{p}}\xi$, we further obtain that
\begin{align*}
   J_{2}\leq CLv^{2(\gamma-1)}\frac{v^{p}}{(1+v^{p})^{2}}\frac{t}{(1+v^{p}t^{m})^{1-\frac{2}{p}}}\int_{0}^{\infty}\frac{\xi \,\der \xi}{1+\xi^{p}}\leq CL v^{2(\gamma-1)}\frac{v^{p}}{(1+v^{p})^{2}}\frac{t }{(1+v^{p}t^{m})^{1-\frac{2}{p}}}.
\end{align*}
Remembering the definition of $H$ and since we have, due to \eqref{Vestimate}, that
\begin{align}\label{lower bound g problematic region}
    H(x,v,t)=\frac{C_{0}}{1+|X(x,v,t)|^{m}+V(x,v,t)^{p}}\geq \frac{C}{1+v^{p}}
\end{align}
we deduce that 
\begin{align}\label{T2est}
   J_{2}\leq \frac{ CL v^{2(\gamma-1)}t}{(1+v^{p}t^{m})^{1-\frac{2}{p}}}H(x,v,t).
\end{align}
We now analyze the term $\frac{ v^{2(\gamma-1)}t}{(1+v^{p}t^{m})^{1-\frac{2}{p}}}$. It holds that
\begin{align}\label{last last step}
    \frac{ v^{2(\gamma-1)}t}{(1+v^{p}t^{m})^{1-\frac{2}{p}}}=\frac{ (v^{\alpha}t)^{\frac{2(\gamma-1)}{\alpha}}t^{1-\frac{2(\gamma-1)}{\alpha}}}{(1+(v^{\alpha}t)^{m})^{1-\frac{2}{p}}}\leq t^{1-\frac{2(\gamma-1)}{\alpha}}\leq t^{-\frac{\gamma-1}{\alpha}}.
\end{align}
It holds that $-\frac{\gamma-1}{\alpha}>-1$ since $\gamma<\alpha+1$ and thus we have that $\textup{e}^{t^{1-\frac{\gamma-1}{\alpha}}}\leq C$. Then, \eqref{step one problematic region supersolution} follows from \eqref{T2est}.

If $v^{\gamma-1}t \leq 1$, we use (\ref{dvVestimate problematic region}) and then (\ref{estimate for the derivative of g supersolution}) becomes 
\begin{align}\label{gamma t small one}
|\partial_{\tilde{v}}H(x,\tilde{v},t)|\leq \frac{CL v^{p-1}}{(1+v^{p})^{2}}.
\end{align}
We can conclude using the same computations as above and by noticing that  $-\frac{(\gamma-1)}{\alpha}>-1$.

Finally, we need to prove that (\ref{step two problematic region supersolution}) holds in the case when $x\in[(1-2\delta) v^{\alpha}t,(1+2\delta)v^{\alpha}t]$. Assume first that $v^{\gamma-1}t \geq 1$. We have that
\begin{align*}
 &\frac{v^{\gamma}}{1+|x|^{m-d}}|\partial_{v}H(x,v,t)|\leq \frac{C v^{\gamma}|\partial_{v}H(x,v,t)|}{1+(v^{\alpha}t)^{m-d}}. 
 \end{align*}
 Making use of (\ref{estimate for the derivative of g supersolution}) and (\ref{lower bound g problematic region}), we further obtain as before that
 \begin{align*}
 &\frac{v^{\gamma}}{1+|x|^{m-d}}|\partial_{v}H(x,v,t)|\leq CL v^{2(\gamma-1)}\frac{v^{p}}{1+v^{p}}\frac{t }{1+(v^{\alpha}t)^{m-d}}H(x,v,t) 
 \end{align*}
 and we can then use similar arguments as in (\ref{last last step}). The case when $v^{\gamma-1}t \leq 1$ can be proved similarly using (\ref{gamma t small one}) instead of (\ref{estimate for the derivative of g supersolution}). This concludes our proof.
\end{proof}
\subsection{Proof of Theorem \ref{main theorem}}\label{subsection existence of limit}

In this subsection, we finish the proof of Theorem \ref{main theorem} by establishing  that there exists a limit for the sequence $\{f_{n}\}_{n\in\mathbb{N}}$ defined in (\ref{inductive sequence}) and then passing to the limit in the equation. Our proof has analogies with the methods used to solve symmetric hyperbolic systems, see for example \cite{majda}.

We first prove some bounds that are independent of $t$ for the function $G$ defined in Definition \ref{definition supersolution}.
\begin{lem}
Let $T>0$ be sufficiently small. Then it holds that
\begin{align}
    G(x,v,t)&\leq \frac{2^{m}C_{0}B_{t}}{1+|x|^{m}+v^{p}}\leq \frac{2^{m+1}C_{0}}{1+|x|^{m}+v^{p}};\label{bound indep of t for g 1}\\
    G(x,v-v',t)&\leq  \frac{2^{m+1}K_{2}C_{0}}{1+|x|^{m}+v^{p}} \textup{ for all } v'\in\bigg(0,\frac{v}{2}\bigg),\label{bound indep of t for g 2}
    \end{align}
 for all $v>0$,  $x\in\mathbb{R}$, and all $t\in[0,T]$, where $B_{t}$ was defined in (\ref{Btdef}), $K_{2}$ is as in Lemma \ref{L.Hproperties}, and $C_{0}$ is as in (\ref{form solution characteristics}).
\end{lem}

\begin{proof}
To prove  \eqref{bound indep of t for g 1} we notice first that it holds that $B_{t}\leq 2$ if we take $t\leq 1$ to be sufficiently small, where $B_{t}$ was defined in \eqref{Btdef}.

We first consider the case $x>0$ and $v\leq v_{\textup{max}}$, where $v_{\textup{max}}(x,t)$ was defined in (\ref{definition v max}).

From Definition \ref{definition supersolution} and the fact that if $v\leq v_{\textup{max}}(x,t)$ then we are in the region where $t\leq \frac{x}{v^{\alpha}(1+2\delta)}$, we can use the bound in (\ref{xpos2}) to deduce that 
\begin{align}
    G(x,v,t)&\leq\frac{C_{0}}{1+|x-(1+\delta)v_{\textup{max}}(x,t)^{\alpha}t|^{m}+(1-\delta)^{p}v_{\textup{max}}(x,t)^{p}},
\end{align}
when $v\leq v_{\textup{max}}(x,t)$, where $C_{0}$ is as in (\ref{form solution characteristics}).

 By \eqref{vmaxestimate}, we have that $\frac{1}{K_{max}}xt^{\frac{1}{m-1}}\leq v_{\textup{max}}(x,t)\leq K_{max} xt^{\frac{1}{m-1}}$ and thus $x\geq x-(1+\delta)v_{\textup{max}}(x,t)^{\alpha}t\geq \frac{x}{2}>0$ when  $x>0$ if $t$ is sufficiently small. Thus, using in addition that $v\leq v_{\textup{max}}$ and that $\delta<\frac{1}{2}$, it holds that $(1-\delta)^{p}v_{\textup{max}}^{p}\geq \frac{v^{p}}{2^{p}}$. Thus,
\begin{align}\label{bound on cut region g}
    G(x,v,t)=G(x,v_{\textup{max}}(x,t),t)&\leq\frac{2^{m}C_{0}}{1+|x|^{m}+v^{p}}.
\end{align}

We now treat the case when $t\in[\frac{x}{(1+2\delta)v^{\alpha}},\frac{x}{(1-2\delta)v^{\alpha}}]$. From \eqref{form solution characteristics} and then using the fact that $x\leq 2v^{\alpha}t\leq v^{\alpha}$, it follows that
\begin{align*}
    G(x,v,t)\leq \frac{2^{p}C_{0}}{1+v^{p}}\leq \frac{2^{p+1}C_{0}}{1+|x|^{m}+v^{p}}.
\end{align*}

If $v^{\alpha}t\leq \frac{x}{1+2\delta}$, from \eqref{xpos2} it holds that
\begin{align*}
    G(x,v,t)\leq \frac{C_{0}}{1+|x-(1+\delta)v^{\alpha}t|^{m}+(1-\delta)^{p}v^{p}}
\end{align*}
and  \eqref{bound indep of t for g 1} follows from (\ref{positive part one}).

   Finally, if $v^{\alpha}t\geq \frac{x}{1-2\delta}$, from \eqref{xpos1} it holds that
\begin{align*}
    G(x,v,t)\leq \frac{C_{0}}{1+|x-(1-\delta)v^{\alpha}t|^{m}+(1-\delta)^{p}v^{p}}
\end{align*}
and  \eqref{bound indep of t for g 1} follows from (\ref{Jestimate}).

 If  $x\leq 0$, from \eqref{xnegative} we have that $|X(x,v,t)|\geq|x-(1-\delta)v^{\alpha}t|=|x|+(1-\delta)v^{\alpha}t$
    and the conclusion follows.

\eqref{bound indep of t for g 2} follows from \eqref{bound indep of t for g 1}, \eqref{new H estimate}, and \eqref{x neg below}. 
\end{proof}
Using these bounds, we can now prove that there exists a limit for the sequence $\{f_{n}\}_{n\in\mathbb{N}}$ that was defined in (\ref{inductive sequence}). 
\begin{prop}\label{sequence has a limit}
Let $T>0$ be sufficiently small. For every $\epsilon>0$, there exists $n_{\epsilon}\in\mathbb{N}$ such that, for every $n,m\geq n_{\epsilon}$, it holds that $||f_{n}-f_{m}||_{\infty}:=\sup_{t\in[0,T], x\in\mathbb{R},v\in(0,\infty)}|f_{n}(x,v,t)-f_{m}(x,v,t)|\leq \epsilon$.
\end{prop}
\begin{proof}
{\bf Step 1: (Set-up)}
Let $t\geq 0$, $x\in\mathbb{R}$ and $v> 0$. Let $\{f_{n}\}_{n\in\mathbb{N}}$ be the sequence defined in (\ref{inductive sequence}). For $n\in\mathbb{N}$, we denote by 
\begin{align}\label{def rn}
R_n(x,v,t):=f_{n+1}(x,v,t)-f_n(x,v,t).
\end{align}
Moreover, for two functions $f,g$, we denote by
\begin{align}\label{k1 term}
\mathbb{K}_{1}[f,g]:=\int_{0}^{\frac{v}{2}}K(v-v',v')f(x,v',t)g(x,v-v',t)\der v',
\end{align}

\begin{align}\label{k2 term}
\mathbb{K}_{2}[f,g]:=\int_{0}^{\infty}K(v,v')f(x,v',t)g(x,v,t)\der v',
\end{align}
and
\begin{align}\label{full k term}
\mathbb{K}[f,g]:=\mathbb{K}_{1}[f,g]-\mathbb{K}_{2}[f,g].
\end{align}
Using this notation, it holds that
\begin{align}\label{induction step with distance one}
\partial_t R_n(x,v,t)+v^{\alpha}\partial_x R_n(x,v,t)&=\mathbb{K}_{1}[f_{n},f_{n+1}]-\mathbb{K}_{1}[f_{n-1},f_{n}]-\mathbb{K}_{2}[f_{n},f_{n+1}]+\mathbb{K}_{2}[f_{n-1},f_{n}]\nonumber\\
&=\mathbb{K}_{1}[f_{n},f_{n+1}]-\mathbb{K}_{1}[f_{n},f_{n}]+\mathbb{K}_{1}[f_{n},f_{n}]-\mathbb{K}_{1}[f_{n-1},f_{n}]\nonumber\\
&-\mathbb{K}_{2}[f_{n},f_{n+1}]+\mathbb{K}_{2}[f_{n},f_{n}]-\mathbb{K}_{2}[f_{n},f_{n}]+\mathbb{K}_{2}[f_{n-1},f_{n}]\nonumber\\
&=\mathbb{K}[f_{n},R_n]+\mathbb{K}[R_{n-1},f_{n}].
\end{align}
\begin{rmk}\label{remark induction step}
    Notice that it suffices to analyze the term $R_{n}$ in order to obtain the statement of Proposition \ref{sequence has a limit}. This is since we can repeat the computations in (\ref{induction step with distance one}) to obtain
    \begin{align}\label{induction step with distance}
\partial_t [f_{m}-f_{n}]+v^{\alpha}\partial_x[f_{m}-f_{n}]=\mathbb{K}[f_{m-1},f_{m}-f_{n}]+\mathbb{K}[f_{m-1}-f_{n-1},f_{n}].
\end{align}
Since the estimates we will prove do not depend on $n,m\in\mathbb{N}$, we can reduce the problem to analyzing $R_{n}$ in order to simplify the notation.
\end{rmk}
Notice the following. From (\ref{induction step with distance one}) and (\ref{initial function}), we have that $R_{n}$ solves the following system
 \begin{equation}\label{duhamel 1}
\left\{\begin{aligned}
\partial_t R_n(x,v,t)+v^{\alpha}\partial_x R_n(x,v,t)&=\mathbb{K}_{1}[f_{n},R_n]+\mathbb{K}_{1}[R_{n-1},f_{n}]-\mathbb{K}_{2}[f_{n},R_{n}]-\mathbb{K}_{2}[R_{n-1},f_{n}]; \\
R_{n}(x,v,0)&=0.
   \end{aligned}\right.
   \end{equation}

Since the system is linear in $R_{n}$, by Duhamel's principle, it suffices to derive estimates for
 \begin{equation}\label{duhamel 2}
\left\{\begin{aligned}
\partial_t R^{s}_n(x,v,t)+v^{\alpha}\partial_x R^{s}_n(x,v,t)&=\mathbb{K}_{1}[f_{n},R^{s}_n]-\mathbb{K}_{2}[f_{n},R^{s}_{n}], \textup{ for } t>s; \\
R^{s}_{n}(x,v,s)&=\mathbb{K}_{1}[R_{n-1},f_{n}](x,v,s)-\mathbb{K}_{2}[R_{n-1},f_{n}](x,v,s).
   \end{aligned}\right.
   \end{equation}

We now prove suitable estimates for the inhomogeneous part in (\ref{duhamel 1}), which in turn will give us suitable estimates for $R_{n}$. 

\bigskip
{\bf Step 2: (Induction basis)}  It holds
\begin{align}\label{R1estimate}
R_1(x,v,t)\leq \frac{CT(v^{\gamma}+1)}{(1+|x|^{m}+v^{p})(1+|x|^{m-d})},
\end{align}
for all $t\in[0,T],$ where $d$ is as in (\ref{we need function to be even}).

We will prove \eqref{R1estimate} after (\ref{induction step for later use}) since the estimates are similar.

\bigskip
{\bf Step 3: (Induction step)}
It holds 
\begin{align}\label{induction step for later use}
    R_n(x,v,t)\leq \frac{(CT)^{n}(v^{\gamma}+1)}{(1+|x|^{m}+v^{p})(1+|x|^{m-d})},
\end{align}
for all $t\in[0,T],$ where $d$ is as in (\ref{we need function to be even}).

We assume by induction that there exists a constant $C>0$ such that

\begin{align}\label{estimate induction step}
    R_{n-1}(x,v,t)\leq \frac{(CT)^{n-1}(v^{\gamma}+1)}{(1+|x|^{m}+v^{p})(1+|x|^{m-d})}.
\end{align}

We estimate the inhomogeneous terms $\mathbb{K}_{1}[R_{n-1},f_{n}]$ and $\mathbb{K}_{2}[R_{n-1},f_{n}]$. Assume that the following inequality holds
\begin{align}\label{suitable estimate inhomogeneous terms}
   | \mathbb{K}_{1}[R_{n-1},f_{n}]-\mathbb{K}_{2}[R_{n-1},f_{n}]|\leq\frac{C^{n}T^{n-1}(v^{\gamma}+1)}{(1+|x|^{m}+v^{p})(1+|x|^{m-d})}.
\end{align}
From (\ref{duhamel 2}), it follows that
\begin{align}\label{boooound}
  R_{n}^{s}(x,v,t)\leq \frac{C^{n}T^{n-1}(v^{\gamma}+1)}{(1+|x|^{m}+v^{p})(1+|x|^{m-d})},
\end{align}
for $t\in[0,T]$, if $T$ is sufficiently small. Since from (\ref{duhamel 1}) and (\ref{duhamel 2}), we have that
\begin{align*}
R_{n}(x,v,t)=\int_{0}^{t}R_{n}^{s}(x,v,t)\der s,
\end{align*}
the conclusion (\ref{induction step for later use}) follows from (\ref{boooound}).

We are thus left to prove that (\ref{suitable estimate inhomogeneous terms}) holds. We start with $\mathbb{K}_{2}[R_{n-1},f_{n}]$. Using Proposition \ref{proposition about supersolution}, (\ref{bound indep of t for g 1}) and the fact that $K(v,v')\leq K_{1}(v^{\gamma}+v'^{\gamma})$ from (\ref{condition on gamma existence}), we have that
\begin{align}
\mathbb{K}_{2}[R_{n-1},f_{n}]&=\int_{0}^{\infty}K(v,v')R_{n-1}(x,v',t)f_{n}(x,v,t)\der v'\leq \frac{Cv^{\gamma}}{1+|x|^{m}+v^{p}}\int_{0}^{\infty}R_{n-1}(x,v',t)\der v'\nonumber\\
&+\frac{C}{1+|x|^{m}+v^{p}}\int_{0}^{\infty}v'^{\gamma}R_{n-1}(x,v',t)\der v'.\label{rn2}
\end{align}
Using assumption (\ref{estimate induction step}) and then (\ref{bound needed for the cauchy sequence}) for $n=0$ and $n=2\gamma$, we further obtain that
\begin{align}
    \int_{0}^{\infty}(1+v'^{\gamma})R_{n-1}(x,v',t)\der v' &\leq  \frac{(CT)^{n-1}}{1+|x|^{m-d}}\int_{0}^{\infty}\frac{(v'^{\gamma}+1)^{2}}{1+|x|^{m}+v'^{p}}\der v'\nonumber\\
    &\leq \frac{C^{n}T^{n-1}}{\Big(1+|x|^{m-d}\Big)\Big(1+|x|^{m-\frac{2\gamma+1}{\alpha}}\Big)}\leq  \frac{C^{n}T^{n-1}}{1+|x|^{m-d}}.\label{rn1}
\end{align}
Plugging (\ref{rn1}) into (\ref{rn2}), we deduce that
\begin{align*}
\mathbb{K}_{2}[R_{n-1},f_{n}]\leq\frac{C^{n}T^{n-1}(v^{\gamma}+1)}{(1+|x|^{m}+v^{p})(1+|x|^{m-d})}.
\end{align*}
Similarly, we can bound the term $\mathbb{K}_{1}[R_{n-1},f_{n}]$. More precisely, since $v'\in(0,\frac{v}{2})$, we make use of  Proposition \ref{proposition about supersolution} and (\ref{bound indep of t for g 2}) and, as before, it holds that
\begin{align*}
\mathbb{K}_{1}[R_{n-1},f_{n}]&=\int_{0}^{\frac{v}{2}}K(v-v',v')R_{n-1}(x,v',t)f_{n}(x,v-v',t)\der v'\\
&\leq  \frac{Cv^{\gamma}}{1+|x|^{m}+v^{p}}\int_{0}^{\infty}R_{n-1}(x,v',t)\der v'\\
&\leq\frac{C^{n}T^{n-1}(v^{\gamma}+1)}{(1+|x|^{m}+v^{p})(1+|x|^{m-d})}.
\end{align*}
This concludes the proof of \eqref{induction step for later use}.

\bigskip
{\bf Step 4: (Proof of \eqref{R1estimate})}
Using the definition of $f_{0}$ in (\ref{definition f0}), we obtain that
  \begin{align*}
     \partial_t R_1(x,v,t)+v^{\alpha}\partial_x R_1(x,v,t)=\mathbb{K}[f_{0},R_{1}]+\mathbb{K}[f_{0},f_{0}].
\end{align*}  
Following the steps of the proof of \eqref{induction step for later use}, we obtain that
\begin{align*}
    \mathbb{K}[f_{0},f_{0}]\leq \frac{C(1+v^{\gamma})}{(1+|x|^{m}+v^{p})(1+|x|^{m-d})}
\end{align*}
and the conclusion follows by Duhamel's principle as before.

\bigskip
{\bf Step 5: (Conclusion)}
We combine Remark \ref{remark induction step} with (\ref{R1estimate}) and (\ref{induction step for later use}) and choose the time $T$ in (\ref{induction step for later use}) to be sufficiently small, such that the right-hand side of (\ref{induction step for later use}) tends to zero as $n\rightarrow \infty$.

\end{proof}
We are now able to conclude the proof  of Theorem \ref{main theorem}.
\begin{proof}[Proof of Theorem \ref{main theorem}]

We first prove that if $t\leq T$ and $T$ is sufficiently small, it holds that
\begin{align}\label{bound for f0}
    f_{0}(x,v,t)\leq \frac{2C_{0}}{1+|x|^{m}+v^{p}}.
\end{align}From (\ref{definition f0}) and (\ref{definition f0 initial condition}) it follows that
$f_{0}(x,v,t)\leq \frac{C_{0}}{1+|x-v^{\alpha}t|^{m}+v^{p}}.$ 
We first consider $x>0$.
If $v^{\alpha}t\leq x\bigg(1-\frac{1}{\sqrt[m]{2}}\bigg)$ then  $x-v^{\alpha}t\geq \frac{x}{\sqrt[m]{2}}$ and thus 
\begin{align*}
          \frac{1}{1+|x-v^{\alpha}t|^{m}+v^{p}}\leq \frac{1}{1+\frac{x^m}{2}+v^p}\leq \frac{2}{1+x^m+v^p}.
    \end{align*}
    If $v^{\alpha}t\geq  x\bigg(1-\frac{1}{\sqrt[m]{2}}\bigg)$ then $t\neq 0$ and since $t \leq 1$ is sufficiently small we have that
$    v^{p}\geq 2x^{m}$ and thus
      \begin{align*}
          \frac{1}{1+|x-v^{\alpha}t|^{m}+v^{p}}\leq \frac{1}{1+v^p}=\frac{1}{1+\frac{v^p}{2}+\frac{v^p}{2}}\leq \frac{1}{1+x^{m}+\frac{v^p}{2}}\leq \frac{2}{1+x^{m}+v^p}.
    \end{align*}
 If  $x\leq 0$, then, since $m$ is even, we have $(x-v^{\alpha}t)^{m}=(|x|+v^{\alpha}t)^{m}\geq x^{m}$
    and the conclusion follows.
    Then, by \eqref{bound for f0} and (\ref{bound needed for the cauchy sequence}), we obtain that 
\begin{align}\label{bound moment f0 final theorem}
    \int_{(0,\infty)}vf_{0}(x,v,t)\der v &\leq \frac{2K_{3}C_{0}}{1+|x|^{m-\frac{2}{\alpha}}}.
 \end{align}

On the other hand, by Lemma \ref{estimate for mass}, it follows that we can find a constant $K_{3}>0$ such that
\begin{align}\label{bound for g final theorem}
\int_{(0,\infty)}v
G (x,v,t)\der v &\leq \frac{K_{3}C_{0}B_t}{1+|x|^{m-d}},
\end{align}
where $d$ was defined in (\ref{we need function to be even}), $B_{t}$ was defined in (\ref{Btdef}), and  $G$ was defined in Definition \ref{definition supersolution}. Moreover, we can choose $t\leq 1$ to be sufficiently small (as in Remark \ref{choice of time}) such that $B_{t}\leq 2$ and thus it holds that $\int_{(0,\infty)}v
G (x,v,t)\der v \leq \frac{2K_{3}C_{0}}{1+|x|^{m-d}}$. 
 
We use induction in order to prove that $f_{n}\leq G$ for all $n\in\mathbb{N}$. By Proposition \ref{proposition about supersolution}, the induction step holds true. For the induction basis, we need to prove that \eqref{main ingredient supersolution} holds true for $n=0$. This is done with the same estimates as in Proposition \ref{proposition about supersolution} by using \eqref{bound moment f0 final theorem}. Thus,  if we take $L$ in (\ref{eq char}) to be as in Definition \ref{definition supersolution}, namely $L=4K_{1}K_{2}K_{3}C_{0}$, where $K_{1}$ is as in \eqref{condition on gamma existence}, $K_{2}$ is as in Lemma \ref{L.Hproperties}, and $K_{3}$ is as in (\ref{bound for g final theorem}), we can conclude that 
\begin{align}\label{main inequality to prove convergence}
    f_{n}(x,v,t)\leq G(x,v,t),
\end{align}
for all $n\in\mathbb{N}$.

From (\ref{main inequality to prove convergence}) and \eqref{bound indep of t for g 1}, it follows that there exists some $C>0$ such that
\begin{align}\label{this is all we need}
    f_{n}(x,v,t)\leq  \frac{C}{1+|x|^{m}+v^{p}},
\end{align}
for all $n\in\mathbb{N}$, $t\in[0,T]$, $x\in\mathbb{R}$, and $v\in(0,\infty)$.

By Proposition \ref{sequence has a limit} we have that the exists a limit of the sequence $\{f_{n}\}_{n\in\mathbb{N}}$. It remains to show that the limit of the sequence $\{f_{n}\}_{n\in\mathbb{N}}$ satisfies  equation (\ref{mild solution equation}). With the bound on $\{f_n\}$ in (\ref{this is all we need}) and the bound on the kernel \eqref{condition on gamma existence} it is completely standard to pass to the limit in the equation. Mass-conservation of $f_{n}$ follows by testing with $v$ in (\ref{inductive sequence}) and then integrating in $v$ and $x$. Mass-conservation of $f$ then follows by passing to the limit as $n\rightarrow\infty$ in $\int_{\mathbb{R}}\int_{(0,\infty)}vf_{n}(x,v,t)\der v \der x.$ We omit the details here.

 \end{proof}

\appendix

\section{Estimates for the second order derivative of $G_{L}$}\label{appendix c}

We will prove in this appendix that (\ref{second derivative is monotone}) holds true.

\begin{prop}\label{0 prop second derivative}
 Given $L>0$ and $\delta\in\big(0,\frac{1}{2}\big)$ there exists a sufficiently large $R>0$ such that for all $t\in[0,T]$, with $T$ sufficiently small, which is independent of $L,\delta,$ and $R$, it holds that 
     \begin{align}\label{monotonicity in v max}
 \partial_{v}^{2}G_{L}(x,v,t)< 0
  \end{align}
if
\begin{align}\label{App2.1}
    \frac{1}{K_{max}} x t^{\frac{1}{m-1}} \leq v^{\alpha}
 \leq K_{max} x t^{\frac{1}{m-1}},
\end{align} where $K_{max}$ is as in (\ref{vmaxestimate}).
\end{prop}
\begin{proof}
Let $Q(x,v,t)=1+|X|^{m}+V^{p}$. Since   $G_{L}=\frac{1}{Q}$, we have that 
\begin{align}
    \partial_{v}Q(x,v,t)&=m|X|^{m-2}X\partial_{v}X+pV^{p-1}\partial_{v}V,\label{App2.2}\\
\partial^{2}_{v}Q(x,v,t)&=m(m-1)|X|^{m-2}|\partial_{v}X|^{2}+p(p-1)V^{p-2}(\partial_{v}V)^{2}+m|X|^{m-2}X\partial^{2}_{v} X+p V^{p-1}\partial^{2}_{v} V,\label{second order for q}
\end{align}
and
\begin{align}\label{main proposition}
\partial_{v}^{2}G_{L}=2\frac{|\partial_{v}Q|^{2}}{Q^{3}}-\frac{\partial_{v}^{2}Q}{Q^{2}}=\frac{2|\partial_{v}Q|^{2}-Q\partial_{v}^{2}Q}{Q^{3}}.
\end{align}

It is worthwhile to notice that if $v$ is as in (\ref{App2.1}), then $x\geq (1+2\delta)v^{\alpha}t$ if $t$ is sufficiently small and thus (\ref{dvVestimate}) holds. We analyze first the term $2|\partial_{v}Q|^{2}$ in (\ref{main proposition}). Since $v^{\alpha}\leq K_{max} xt^{\frac{1}{m-1}}$, we have that $0\leq x-(1+\delta)v^{\alpha}t\leq X\leq x-(1-\delta)v^{\alpha}t\leq x$ and  $x>0$. Using Proposition \ref{properties of the characteristics} and \eqref{App2.1} we deduce that
\begin{align*}
    |\partial_{v}Q|\leq C_{1}x^{m-1}v^{\alpha-1}t+C_{2}v^{p-1}\leq \frac{Cx^{m}}{v}t^{\frac{m}{m-1}},
\end{align*}
for some constants $C_{1},C_{2},C>0$, and thus
     \begin{align}
  2|\partial_{v}Q(x,v,t)|^{2}\leq  \frac{Cx^{2m}}{v^{2}}t^{\frac{2m}{m-1}}.\label{b2}
\end{align}
 We then analyze the terms in (\ref{second order for q}). We have $|X|^{m-2} |\partial_v X|^2 \geq 0$ and from Proposition \ref{properties of the characteristics} we know 
 \begin{align}\label{App 2.5}
     V^{p-2}(\partial_{v}V)^{2}
     \geq \frac{1}{C}v^{p-2}.
\end{align}
 We are going to show that if $t$ is sufficiently small  and $v$ as in \eqref{App2.1} we have 
 \begin{align}\label{App2.4}
  \partial^2_v X (x,v,t) \geq 0
 \end{align}
and that for any $\epsilon >0$ and $v$ as in \eqref{App2.1} we have 
\begin{align}
    |V^{p-1}\partial^{2}_{v}V|&\leq \epsilon v^{p-2} 
    \label{App 2.3}
\end{align} 
if $R$ is sufficiently large and $t$ is sufficiently small. Combining then (\ref{App 2.5})-(\ref{App 2.3}) with (\ref{App2.1}) and noticing that $Q \geq |X|^m\geq (x-(1+\delta)v^{\alpha}t)^{m}\geq \frac{x^{m}}{2^{m}}$ if $t$ is sufficiently small, it follows that 
\begin{align}\label{b2 2}
   Q\partial_{v}^{2}Q\geq \frac{x^{m}}{C} v^{p-2}\geq \frac{x^{2m}}{Cv^{2}}t^{\frac{m}{m-1}}.
\end{align}
(\ref{main proposition}), (\ref{b2}) and (\ref{b2 2}) imply that we can find a sufficiently small  $T\in(0,1)$  such that (\ref{monotonicity in v max}) holds for all $t\leq T$.

In all the following computations, we will assume for simplicity that $v\geq 2R$, with $R$ as in Proposition  \ref{properties of the characteristics}. The rest of the cases can be proved using similar computations.

\medskip
{\bf Auxiliary result:} 
Given $L>0$ and $\delta\in\big(0,\frac{1}{2}\big)$ there exists a sufficiently large $R>0$ such that for all $t\in[0,T]$, with $T$ sufficiently small, which is independent of $L,\delta,$ and $R$, it holds that 
\begin{align}\label{App2.6}
|\partial^{2}_{v}X(x,v,t)|\leq C v^{\alpha-2}t
\end{align}
if $v$ is as in (\ref{App2.1}).

In order to show \eqref{App2.6} we first prove that there exists  $C>0$ such that 
 \begin{align}\label{der in vv of phi}
|\partial^{2}_{v}\Phi(z,v)|\leq C v^{-1}, \textup{ for } v\geq 2R,
\end{align}
 where $\Phi$ is as (\ref{Phiequation}). We differentiate two times with respect to $v$ in (\ref{Phi-2}) in order to  deduce that 
\begin{align*}    \partial^{2}_{v}\Phi(z,v)&=(\gamma{-}\alpha)\bigg(v^{1-(\gamma-\alpha)}+(1-(\gamma {-}\alpha))z\bigg)^{\frac{2(\gamma-\alpha)-1}{1-(\gamma-\alpha)}}v^{-2(\gamma-\alpha)}\\
    &\qquad -(\gamma{-}\alpha)(v^{1-(\gamma-\alpha)}+(1-(\gamma{-}\alpha))z)^{\frac{\gamma-\alpha}{1-(\gamma{-}\alpha)}}v^{-(\gamma-\alpha)-1}.
    \end{align*}
Since $0\leq z \leq C$ and we are in the case when $v\geq 2R$, we can choose $R>0$ sufficiently large such that 
$    v^{1-(\gamma-\alpha)}\leq v^{1-(\gamma-\alpha)}+(1-(\gamma-\alpha))z\leq 2 v^{1-(\gamma-\alpha)}$
and thus $|\partial^{2}_{v}\Phi(z,v)|\leq C v^{-1}$.

In order to prove \eqref{App2.6} we differentiate (\ref{tintegral}) twice with respect to $v$ and obtain 
\begin{align}\label{two diff in x}
&\frac{\partial^{2}_{v}X(x,v,t) }{\Phi(\psi(X(x,v,t))-\psi(x),v)^{\alpha}}-\alpha\frac{\partial_{v}X(x,v,t)[\partial_{v}\Phi+\partial_{z}\Phi \psi'(X)\partial_{v}X] }{\Phi(\psi(X(x,v,t))-\psi(x),v)^{\alpha+1}}\nonumber\\
&=  \alpha\int_{x}^{X} \frac{\partial^{2}_{v}\Phi\der \xi}{\Phi(\psi(\xi)-\psi(x),v)^{\alpha+1}}+ \alpha(\alpha+1)\int_{X}^{x} \frac{(\partial_{v}\Phi)^{2}\der \xi}{\Phi(\psi(\xi)-\psi(x),v)^{\alpha+2}}\nonumber\\
&\qquad +  \frac{\alpha\partial_{v}\Phi}{\Phi(\psi(X)-\psi(x),v)^{\alpha+1}}\partial_{v}X.
\end{align}
From (\ref{der in vv of phi}), \eqref{Xestimate} and \eqref{Vestimate}  we deduce that 
\begin{align}\label{estimate 1 of 4}
\bigg|\int_{x}^{X} \frac{\partial^{2}_{v}\Phi\der \xi}{\Phi(\psi(\xi)-\psi(x),v)^{\alpha+1}}\bigg|\leq Cv^{-\alpha-2}[x-X]\leq C v^{-2}t.
  \end{align}
Furthermore,  Proposition \ref{properties of the characteristics} implies that
  \begin{align}\label{estimate 2 of 4}
\frac{|\partial_{v}\Phi\partial_{v}X|}{\Phi(\psi(X)-\psi(x),v)^{\alpha+1}}&\leq C v^{-2}t\,,\\
\label{estimate 3 of 4}
\int_{X}^{x} \frac{(\partial_{v}\Phi)^{2}\der \xi}{\Phi(\psi(\xi)-\psi(x),v)^{\alpha+2}}\leq Cv^{-\alpha-2}[x-X]&\leq C v^{-2}t
\end{align}
and that
$\frac{1}{C} v^{\alpha} \leq \Phi(\psi(X)-\psi(x),v)^{\alpha}\leq C v^{\alpha}.$

We claim that  if $v$ satisfies \eqref{App2.1}  it holds 
\begin{align}\label{the annoying term}
 0\leq \frac{\partial_{z}\Phi \psi'(X)|\partial_{v}X|^{2}}{\Phi(\psi(X)-\psi(x),v)^{\alpha+1}}\leq C v^{-2}t.
\end{align}
Indeed, Proposition  \ref{properties of the characteristics} implies that $|\partial_{v}X|\leq C  v^{\alpha-1}t$ and  $\frac{1}{\Phi(\psi(X)-\psi(x),v)^{\alpha+1}}\leq C v^{-\alpha-1}$.

We use \eqref{App2.1}  again to deduce that we are in the region where $x-(1+2\delta)v^{\alpha}t\geq 0$ for sufficiently small $t$. From the estimate (\ref{needed for the appendix}) we obtain that
\begin{align}\label{this is changed}
0\leq \partial_{z}\Phi \psi'(X)|\partial_{v}X|\leq C.
\end{align}
Thus, (\ref{the annoying term}) follows.

Combining the estimates (\ref{estimate 1 of 4})-(\ref{the annoying term}) and then making use of (\ref{two diff in x}), we obtain \eqref{App2.6}.

\medskip
{\bf Proof of \eqref{App 2.3}:}

 We only look at the case when $\gamma>1$. The case $\gamma\in[0,1]$ can be proved in a similar manner. As before, we assume in the following for simplicity that the constant $L$ in (\ref{the original characteristics}) is $L=1$. 
 If $\gamma\in(1,1+\alpha)$, then by integrating in \eqref{the original characteristics} it follows that
  \begin{align}\label{main equlity for volume}
      V(x,v,t)^{1-\gamma}-v^{1-\gamma}=(\gamma-1)\int_{0}^{t}\frac{\der \xi}{1+|X(x,v,\xi)|^{m-d}}
  \end{align}
such that 
\begin{align}\label{form of v}
      V(x,v,t)=\frac{v}{\bigg(1+(\gamma-1)v^{\gamma-1}\int_{0}^{t}\frac{\der \xi}{1+|X|^{m-d}}\bigg)^{\frac{1}{\gamma-1}}}.
  \end{align}
  Differentiating in $v$, we obtain that 
     \begin{align}\label{derivate in v of vv}
     \partial_{v} V(x,v,t)&=\frac{1}{\bigg(1+(\gamma-1)v^{\gamma-1}\int_{0}^{t}\frac{\der \xi}{1+|X|^{m-d}}\bigg)^{\frac{1}{\gamma-1}}}-\frac{(\gamma-1)v v^{\gamma-2}\int_{0}^{t}\frac{\der \xi}{1+|X|^{m-d}}}{\bigg(1+(\gamma-1)v^{\gamma-1}\int_{0}^{t}\frac{\der \xi}{1+|X|^{m-d}}\bigg)^{\frac{\gamma}{\gamma-1}}}\\
     &+\frac{v v^{\gamma-1}\int_{0}^{t}\frac{(m-d)|X|^{m-d-2}X\partial_{v}X\der \xi}{(1+|X|^{m-d})^{2}}}{\bigg(1+(\gamma-1)v^{\gamma-1}\int_{0}^{t}\frac{\der \xi}{1+|X|^{m-d}}\bigg)^{\frac{\gamma}{\gamma-1}}}
     =:\frac{T_{1}}{T_{2}^{\frac{\gamma}{\gamma-1}}},\nonumber
  \end{align}
  where
  \begin{align*}
      T_{1}:=1+v^{\gamma}\int_{0}^{t}\frac{(m-d)|X|^{m-d-2}X\partial_{v}X\der \xi}{(1+|X|^{m-d})^{2}}
  \end{align*}
  and
\begin{align*}
      T_{2}:=1+(\gamma-1)v^{\gamma-1}\int_{0}^{t}\frac{\der \xi}{1+|X|^{m-d}}.
  \end{align*}
 
  Differentiating in $v$, we obtain that 
  \begin{align*}
   \partial^{2}_{v} V(x,v,t)
     =\frac{\partial_{v}T_{1}}{T_{2}^{\frac{\gamma}{\gamma-1}}}-\frac{\gamma}{\gamma-1}\frac{T_{1}\partial_{v}T_{2}}{T_{2}^{\frac{2\gamma-1}{\gamma-1}}}.
  \end{align*}
  It holds that
  \begin{align*}
      \partial_{v}T_{1}&=\gamma v^{\gamma-1}\int_{0}^{t}\frac{(m-d)|X|^{m-d-2}X\partial_{v}X\der \xi}{(1+|X|^{m-d})^{2}}-2v^{\gamma}\int_{0}^{t}\frac{\big|(m-d)|X|^{m-d-2}X\partial_{v}X\big|^{2}\der \xi}{(1+|X|^{m-d})^{3}}\\
      &\qquad+v^{\gamma}\int_{0}^{t}\frac{(m-d)|X|^{m-d-2}X\partial^{2}_{v}X\der \xi}{(1+|X|^{m-d})^{2}}+v^{\gamma}\int_{0}^{t}\frac{(m-d)(m-d-1)|X|^{m-d-2}|\partial_{v}X|^{2}\der \xi}{(1+|X|^{m-d})^{2}}.
  \end{align*}
  Moreover, we have that
  \begin{align*}
\partial_{v}T_{2}=(\gamma-1)^{2}v^{\gamma-2}\int_{0}^{t}\frac{\der \xi}{1+|X|^{m-d}}-(\gamma-1)v^{\gamma-1}\int_{0}^{t}\frac{(m-d)|X|^{m-d-2}X\partial_{v}X\der \xi}{(1+|X|^{m-d})^{2}}.
  \end{align*}

 We  first prove the following estimates:
   \begin{align}\label{we need this}
0\leq v^{\gamma-1}\int_{0}^{t}\frac{\der \xi}{1+|X(x,v,\xi)|^{m-d}}\leq \frac{1}{2}
\end{align}
and
\begin{align}\label{we thus deduce that}
   v^{\gamma}\bigg|\int_{0}^{t}\frac{(m-d)|X(x,v,\xi)|^{m-d-2}X(x,v,\xi)\partial_{v}X(x,v,\xi)\der \xi}{(1+|X(x,v,\xi)|^{m-d})^{2}} \bigg|\leq \frac{1}{2}.
\end{align}

 We first prove that (\ref{we need this}) holds. For sufficiently small $t$, we are in the case when $x\geq(1+2\delta)v^{\alpha}t$ since $v$ is as  in (\ref{App2.1}).  We use \eqref{Xestimate} and the fact that $x\geq (1+2\delta)v^{\alpha}t\geq(1+2\delta)v^{\alpha}s$ for all $s\in[0,t]$. Thus,  $|X(s)|\geq |x-(1+\delta)v^{\alpha}s|$ when $x\geq(1+2\delta)v^{\alpha}t$, for all $s\in[0,t]$. It follows that 
\begin{align}\label{we need this for later part 1}
v^{\gamma-1}\int_{0}^{t}\frac{\der \xi}{1+|X(x,v,\xi)|^{m-d}}&\leq C v^{\gamma-1}\int_{0}^{t}\frac{\der \xi}{1+|x-(1+\delta)v^{\alpha}\xi|^{m-d}}\nonumber\\
 &\leq C v^{\gamma-\alpha-1}\int_{0}^{v^{\alpha}t}\frac{\der z}{1+|x-(1+\delta)z|^{m-d}}\leq C v^{\gamma-\alpha-1},
\end{align}
which, since $v\geq R$ and $\gamma<\alpha+1$, implies \eqref{we need this} if $R$ is sufficiently large.


We now prove (\ref{we thus deduce that}). As before, we consider the case when $x\geq (1+2\delta)v^{\alpha}t$ since we can choose $t$ sufficiently small.  It follows that 
\begin{align*}
v^{\gamma}\bigg|\int_{0}^{t}\frac{(m-d)|X|^{m-d-2}X\partial_{v}X\der \xi}{(1+|X|^{m-d})^{2}}\bigg|&\leq Cv^{\gamma-1}\int_{0}^{t}\frac{|x-(1-\delta)v^{\alpha}\xi|^{m-d-1}v^{\alpha}\xi\der \xi}{(1+|x-(1+\delta)v^{\alpha}\xi|^{m-d})^{2}}.
\end{align*}
If $v$ is as in (\ref{App2.1}), then $x-(1+\delta)z\geq x-2z\geq \frac{x}{2}$ for sufficiently small $t$, for $z\in[0,v^{\alpha}t]$. By making the change of variables $z=v^{\alpha}\xi$ and using the fact that $x-(1-\delta)z\leq x$ we further obtain that
\begin{align*}
 v^{\gamma-1}\int_{0}^{t}\frac{|x-(1-\delta)v^{\alpha}\xi|^{m-d-1}v^{\alpha}\xi\der \xi}{(1+|x-(1+\delta)v^{\alpha}\xi|^{m-d})^{2}}&\leq C v^{\gamma-\alpha-1} \int_{0}^{v^{\alpha}t}\frac{|x-(1-\delta)z|^{m-d-1}z\der z}{(1+|x-(1+\delta)z|^{m-d})^{2}}\\
&\leq C v^{\gamma-\alpha-1} \int_{0}^{\frac{x}{1+2\delta}}\frac{|x|^{m-d}\der z}{(1+|x|^{m-d})^{2}}\leq  \frac{C v^{\gamma-\alpha-1} |x|^{m-d+1}}{(1+|x|^{m-d})^{2}}\\
&\leq  C v^{\gamma-\alpha-1},
\end{align*}

and thus  (\ref{we thus deduce that})  follows if we take $R$ sufficiently large since $v\geq R$ and $\gamma<\alpha+1$.

From (\ref{we need this}) and (\ref{we thus deduce that}), it holds that $T_{1}\in[\frac{1}{2},1]$  and
$T_{2}\in[\frac{1}{2},\frac{3}{2}]$. From this we deduce that
\begin{align*}
|\partial^{2}_{v}V(x,v,t)|\leq C \big( |\partial_{v}T_{1}|+|\partial_{v}T_{2}|\big)\,.
\end{align*}
We analyze each term of $\partial_{v}T_{1}$ separately. Following the computations for (\ref{we thus deduce that}), we deduce that
\begin{align}\label{1 out of 6}
    v^{\gamma-1}\bigg|\int_{0}^{t}\frac{(m-d)|X|^{m-d-2}X\partial_{v}X\der \xi}{(1+|X|^{m-d})^{2}}\bigg|\leq C   v^{\gamma-\alpha-1}v^{-1}.
\end{align}

Similarly, by making the change of variables $z=v^{\alpha}\xi$, it follows that
\begin{align}\label{2 out of 6}
v^{\gamma}\int_{0}^{t}\frac{\big||X|^{m-d-2}X\partial_{v}X\big|^{2}\der \xi}{(1+|X|^{m-d})^{3}}&\leq C v^{\gamma}\int_{0}^{t}\frac{\big(|x-(1-\delta)v^{\alpha}\xi|^{m-d-1}v^{\alpha-1}\xi\big)^{2}\der \xi}{(1+|x-(1+\delta)v^{\alpha}\xi|^{m-d})^{3}}\nonumber\\
&\leq C v^{\gamma-\alpha-2}\int_{0}^{v^{\alpha}t}\frac{\big(|x-(1-\delta)z|^{m-d-1}z\big)^{2}\der z}{(1+|x-(1+\delta)z|^{m-d})^{3}} \leq C v^{\gamma-\alpha-1}v^{-1}.
\end{align}

Moreover, \eqref{App2.6} implies  that
\begin{align}\label{3 out of 6}
v^{\gamma}\bigg|\int_{0}^{t}\frac{(m-d)|X|^{m-d-2}X\partial^{2}_{v}X\der \xi}{(1+|X|^{m-d})^{2}}\bigg|&\leq  C v^{\alpha-2}v^{\gamma}\int_{0}^{t}\frac{(m-d)|X|^{m-d-1}\xi\der \xi}{(1+|X|^{m-d})^{2}}\nonumber\\
&\leq C v^{-\alpha-2}v^{\gamma}\int_{0}^{v^{\alpha}t}\frac{|x-(1-\delta)z|^{m-d-1}z\der z}{(1+|x-(1+\delta)z|^{m-d})^{2}}\nonumber\\
&\leq C v^{\gamma-\alpha-1}v^{-1}.
\end{align}

We now analyze the last term in $\partial_{v}T_{1}$. We have that
\begin{align}\label{4 out of 6}
    v^{\gamma}\int_{0}^{t}\frac{|X|^{m-d-2}|\partial_{v}X|^{2}\der \xi}{(1+|X|^{m-d})^{2}}\leq C v^{\gamma-\alpha-2}\int_{0}^{v^{\alpha}t}\frac{|x-(1-\delta)z|^{m-d-2}z^{2}\der z}{(1+|x-(1+\delta)z|^{m-d})^{2}}\leq C  v^{\gamma-\alpha-1}v^{-1}.
\end{align}

We continue by finding upper bounds for each term in $\partial_{v}T_{2}$. From (\ref{we need this for later part 1}), it holds that 
\begin{align}\label{5 out of 6}
    v^{\gamma-2}\int_{0}^{t}\frac{\der \xi}{1+|X|^{m-d}}\leq C v^{\gamma-\alpha-1}v^{-1},
\end{align}
which offers the right estimate in order to prove  \eqref{App 2.3}. Notice that the last term in $\partial_{v}T_{2}$, which is $v^{\gamma-1}\int_{0}^{t}\frac{|X|^{m-d-2}X\partial_{v}X\der \xi}{(1+|X|^{m-d})^{2}}$, has already been analyzed in (\ref{1 out of 6}).

Combining (\ref{1 out of 6})-(\ref{5 out of 6}) together with the fact that $v\geq R$ and that $\gamma<\alpha+1$, we obtain that
\begin{align*}
    |v^{p-1}\partial_{v}^{2}V(x,v,t)|\leq C  v^{\gamma-\alpha-2}v^{p-1}\leq \epsilon v^{p-2},
\end{align*}
if $v \geq R$ is sufficiently large. This  concludes the proof of \eqref{App 2.3}.

\medskip
{\bf Proof of \eqref{App2.4}:}

 From  (\ref{the original characteristics}) it follows that
 \begin{align}\label{how does x behave}
\partial_{v}^{2}X(x,v,t)=\alpha(1-\alpha)\int_{0}^{t}V^{\alpha-2}(\partial_{v}V)^{2}\der s-\alpha\int_{0}^{t}V^{\alpha-1}\partial_{v}^{2}V\der s.
 \end{align}
 Now \eqref{App 2.3} implies that the second term on the right hand side can be absorbed into the first one and thus \eqref{App2.4} follows.
\end{proof}

\subsection*{Acknowledgements}
 The authors gratefully acknowledge the financial support of the collaborative
research centre The mathematics of emerging effects (CRC 1060, Project-ID 211504053) and of the Bonn International Graduate School of Mathematics at the Hausdorff Center for Mathematics (EXC 2047/1, Project-ID 390685813) funded through the Deutsche Forschungsgemeinschaft (DFG, German Research Foundation).

\subsubsection*{Statements and Declarations}

\textbf{Conflict of interest} The authors declare that they have no conflict of interest.

\textbf{Data availability} Data sharing not applicable to this article as no datasets were generated or analysed during the current study.

\printbibliography[heading=bibintoc]

\end{document}